\documentclass[reqno,10pt]{amsart}
\usepackage{amssymb}
\usepackage{latexsym}
\usepackage{hyperref}
\usepackage{amsmath}
\usepackage{amscd}
\usepackage{color}
\usepackage{enumerate}
\usepackage{amsfonts}
\usepackage{graphicx}
\usepackage{mathrsfs}
\usepackage{pifont}
\usepackage{setspace}

\usepackage{tabu}
\newcommand{\eps}{{\varepsilon}}
\theoremstyle{plain}
\newtheorem{theorem}{Theorem}

\newtheorem{lemma}[theorem]{Lemma}

\theoremstyle{definition}
\newtheorem{definition}[theorem]{Definition}
\newtheorem{remark}[theorem]{Remark}

\numberwithin{equation}{section}
\numberwithin{theorem}{section}

\def\ge{\geqslant}

\def\leq{\leqslant}

\begin{document}

\title[extremal hypersurface equation in Minkowski space]
{Global well-posedness for radial extremal hypersurface equation in $\left(1+3 \right)$-dimensional Minkowski space-time in critical Sobolev space}	

\author[S. Wang]{Sheng Wang}
\address{Shanghai Center for Mathematical Sciences, Fudan University, Shanghai 200433, P.R. China}
\email{19110840011@fudan.edu.cn}

\author[Y. Zhou]{Yi Zhou}
\address{School of Mathematics Science, Fudan University, Shanghai 200433, P.R. China}
\email{yizhou@fudan.edu.cn}

\date{\today}
\maketitle

\begin{abstract}
	In this article, we  prove the global well-posedness in the critical Sobolev space $H_{rad}^2\left(\mathbb{R}^2\right) \times H_{rad}^1 \left(\mathbb{R}^2\right)$ for the radial time-like extremal hypersurface equation in $\left(1+3\right)$- dimensional Minkowski space-time. This  is achieved by deriving a new div-curl type lemma and combined it with energy and ``momentum" balance law to get some space-time estimates of the nonlinearity.
\end{abstract}

\section {Introduction}
 It is well-known that extremal surface in Minkowski space is the $C^2$ surface with vanishing mean curvature under the Minkowski metric. Moreover, the time-like extremal surface in Minkowski space is an important model in Lorentzian geometry and nontrivial example of membrane in field theory.(see Hoppe \cite{Hoppe 1994})

Let $\left(t,x_1,x_2, y\right)$ be any point in the $\left(1+3\right)$-dimensional Minkowski space-time $\left(\mathbb{R}^{1+3}, m\right)$, where $m=diag\left(-1,1,1,1\right)$
is the Minkowski metric. An  area functional can be introduced as follow:
\begin{equation}\label{af}
	\mathbb{J}\left(\phi\right):=\int_{\mathbb{R}_+ \times \mathbb{R}^2} \sqrt{1+\lvert \nabla_x \phi \rvert^2 -\lvert \phi_{t}\rvert^2} dxdt, 
\end{equation}
 where $x=\left(x_1, x_2\right)$.
 A time-like hypersurface $y= \phi\left(t,x_1,x_2\right)$ is called extremal (or minimal), if $\phi$ is the  critical point of the area functional \eqref{af}. 
Some calculus of variations for the area functional \eqref{af} derives the Euler-Lagrange equation
\begin{equation}\label{yfc}
	\left(\frac{\phi_{t}}{\sqrt{1+\lvert \nabla_x \phi \rvert^2 -\lvert \phi_{t}\rvert^2}}\right)_t-\nabla_x\cdot\left(\frac{\nabla_x\phi}{\sqrt{1+\lvert \nabla_x \phi \rvert^2 -\lvert \phi_{t}\rvert^2}}\right)=0.
\end{equation}
The assumption that the hypersurface is time-like, i.e. $1+\lvert \nabla_x \phi \rvert^2 -\lvert \phi_{t}\rvert^2 > 0$, 
guarantees the area functional \eqref{af} and the extremal hypersurface equation \eqref{yfc} makes sense.

In this article, we will consider time-like extremal hypersurface with radial symmetry $\phi\left(t,x\right)=\phi \left(t,r\right), r=\lvert x \rvert$. Thus the area functional is 
\begin{equation}
	\mathbb{J} \left(\phi\right):= 2\pi \int_{\mathbb{R}_+ \times \mathbb{R}^2} r \sqrt{1+\lvert  \phi_r \rvert^2 -\lvert \phi_{t}\rvert^2} drdt.
\end{equation}
Therefore, the corresponding Euler-Lagrange equation becomes 
\begin{equation}\label{equ}
(\frac{r \phi_t }{\Delta^{1/2} })_t - (\frac{r \phi_r }{\Delta^{1/2}})_r=0,
\end{equation}
where  
\begin{equation*}
\Delta :=\left(1+\phi_r^2-\phi_t^2\right).
\end{equation*}

We can also rewrite the extremal hypersurface equation \eqref{equ} as the following quasilinear wave equation, 
\begin{equation}
r\left(1-\phi^2_{t}\right)\phi_{rr}=r\phi_{tt}\left(1+\phi^2_r\right)-2r\phi_t\phi_r\phi_{tr}-\phi_r \Delta.
\end{equation}

Consider the initial value problem of general quasilinear wave equation as follows
\begin{equation}\label{qlus}
\left\{
\begin{aligned}
& g^{ij}\left(\partial u\right) \partial_i\partial_j u= q^{ij}\left(\partial u\right) \partial_i\partial_j u,\\
& t=0: u=u_0 \in H^s\left(\mathbb{R}^d\right), u_{t}=u_1 \in H^{s-1}\left(\mathbb{R}^d\right),
\end{aligned}
\right. 
\end{equation}
where the metric matrix $g^{ij}$ as well as its inverse $g_{ij}$ satisfies the hyperbolicity condition. $g^{ij}\left(\partial u\right)$, $g_{ij}\left(\partial u\right)$
and $q^{ij}\left(\partial u\right) $ are the bounded smooth functions of $\partial u$ and have global bounded derivatives.

Since $u_{\lambda}\left(t,x\right):= \lambda u\left(\frac{t}{\lambda},\frac{x}{\lambda}\right)$ solves \eqref{qlus} for any $\lambda >0$, provides that $u\left(t,x\right)$ is a solution. This fact indicates the critical homogeneous Sobolev space $\dot{H}^{s_c}$ with
\begin{equation}
	s_c = \frac{d+2}{2},
\end{equation}
which provides a lower bound of the regularity of the problem to be well-posed in $H^s$. At the same time, according to the propagation of singularities along the light cone, the ill-posedness for this problem at the regularity level $s\leq s_l= \frac{d+9}{4}$. (see Lindblad \cite{Linblad 1998})

The extremal hypersurface equations can be regarded as a special kind of quasilinear wave equations with nice structure of nonlinearity. In (1+2)- dimensional Minkowski space the problem is fully solved by Kong-Sun-Zhou in \cite{Kong 2006} and \cite{Sun 2006}.  In higher dimension, Brendle \cite{Brendle 2002}, Lindblad \cite{Lindblad 2004} proved that the general minimal surface equations admit global smooth solutions for small initial data.  And it is generalized to any co-dimensions by Allen et al. in \cite{Allen 2006}.  Liu-Zhou \cite{Liu 2019} studied the traveling waves to the time-like extremal
hypersurface in Minkowski space. Besides, the global existence of smooth solutions for small initial data to Born-Infeld equations has been obtained by Chae and Huh in \cite{Chae 2003}. These results are all about classical solutions with small initial data. For the low regularity case, Smith-Tataru \cite{Smith 2005} proved when initial data $\left(\phi_0, \phi_1 \right) \in  H^s\left(\mathbb{R}^2\right) \times H^{s-1} \left(\mathbb{R}^2\right)$ for $ s>2+\frac{3}{4}$  and $\left(\phi_0, \phi_1 \right) \in  H^s\left(\mathbb{R}^d\right) \times H^{s-1} \left(\mathbb{R}^d\right) $ for $s>\frac{d+2}{2}+\frac{1}{2}$ with $ 3\leq d \leq 5$, the quasilinear waves \eqref{qlus} is local well-posedness. To go below $s_l$ need to explore the null structure of the quasilinear wave equation. Based on this, recently, Ai-Ifrim-Tataru \cite{Ai 2021} improve the well-posedness index to $s>2+\frac{3}{8}$  in two space dimensions and $s>\frac{d+2}{2}+\frac{1}{4} $ in higher dimensions for the the time-like extremal hypersurface equation in Minkowski space-time. 

In this paper, we will consider the following Cauchy problem for radially symmetrical solutions to the extremal hypersurface equations in (1+3)-dimensional Minkowski space
\begin{equation}\label{equs}
\left\{
\begin{aligned}
& (\frac{r \phi_t }{\Delta^{1/2} })_t - (\frac{r \phi_r }{\Delta^{1/2}})_r=0,\\
& t=0: \phi=\phi_0, \phi_{t}=\phi_1,
\end{aligned}
\right. 
\end{equation}
with initial data
\begin{equation}\label{ind}
\phi_0\in	H^s\left(\mathbb{R}^2\right),\qquad \phi_1\in H^{s-1} \left(\mathbb{R}^2\right).
\end{equation}

\begin{definition}
	We call a function $\phi\left(t,r\right)$ is a strong solution with initial data $\left(\phi_0,\phi_1\right) \in H_{rad}^2\left(\mathbb{R}^2\right) \times H_{rad}^1 \left(\mathbb{R}^2\right) $ on the time interval $\left[0, T^*\right]$, if there exists sequences of smooth functions $\left(\phi^{\left(n\right)}_0, \phi^{\left(n\right)}_1\right)$ and smooth solutions $\phi^{\left(n\right)}$ of system \eqref{equs} on the time interval $\left[0, T^*\right]$ with $\left(\phi^{\left(n\right)}_0, \phi^{\left(n\right)}_1\right)$ as initial data such that as $n \rightarrow \infty$,  $\left(\phi^{\left(n\right)}_0, \phi^{\left(n\right)}_1\right)$ converges strongly to $\left(\phi_0, \phi_1\right)$ in $ H_{rad}^2\left(\mathbb{R}^2\right) \times H_{rad}^1 \left(\mathbb{R}^2\right) $ and $\phi^{\left(n\right)}$
	converges weakly $-*$ to $\phi$ in $L^{\infty}\left(\left[0, T^*\right], H_{rad}^2\left(\mathbb{R}^2\right)\right) \cap W^{1,\infty}\left(\left[0, T^*\right], H_{rad}^1\left(\mathbb{R}^2\right)\right) $.
\end{definition}
\begin{definition}
	We define the norm of $H_{rad}^2\left(\mathbb{R}^2\right) \times H_{rad}^{1}\left(\mathbb{R}^2\right)$ as follow
	\begin{equation}
	\|\left(\phi_0, \phi_1\right)\|^2_{H_{rad}^2\left(\mathbb{R}^2\right) \times H_{rad}^{1}\left(\mathbb{R}^2\right)}:=\int r\phi^2_{0rr}+\frac{\phi^2_{0r}}{r} + r\phi^2_{_{1r}} + \frac{\phi^2_1}{r} dr 
	\end{equation}
\end{definition}


For the traditional work, the authors usually focused on the energy leval to study the well-posedness. In this work, we observe some nonlinear structure for the time-like extremal hypersurface \eqref{equ} associated with null-form and derive a new div-curl type lemma that connects energy and ``momentum". In fact, the div-curl type lemma was studied by many famous mathematicians, since this type of lemma is a powerful tool to study the nonlinear partial differential equation. Murat \cite{Murat 1979} and Tatar \cite{Tatar 1979 } proved a class of div-curl type lemma associated with weekly convergence which was used to study the entropy weak solution of conservation law (see Dafermos \cite{Dafermos 2010}). Coifman-Lions-Meyer-Semmes studied a class of div-curl type lemma in Hardy space which was used to study the elliptic equations and wave mapping (\cite{HE 2002}). In our method, we can use the new div-curl type lemma to study some balance law (i,e. $f_t + f_r = f^{re}$).  Particularly, we can use the new div-curl lemma to fully explore the energy balance law and ``momentum" balance law. Interestingly, in our approch, the ``momentum" balance law is as important as the energy balance law, this gives a new perspective in the study of the estimates of the nonlinearity.  Based  on this, we obtain the global well-posedness in critical Sobolev space from a new perspective.  Finally, it is worth mentioning that our method can be applied to  Waves, Schr\"odinger and other dispersive equations only if energy and momentum are conserved.

Now we will state our main result as follows.
\begin{theorem}\label{1.3}
	Consider Cauchy problem \eqref{equs} \eqref{ind} for  radially symmetrical solutions to the extremal hypersurface equations in (1+3)-dimensional Minkowski space-time. Under the assumption $\phi_{t}\big |_{r=0} =0$, when the initial data $\left(\phi_0, \phi_1\right) \in H^s\left(\mathbb{R}^2\right) \times H^{s-1} \left(\mathbb{R}^2\right)$  with $s\ge 3 $, and $\|\left(\phi_0, \phi_1\right)\|_{H_{rad}^2\left(\mathbb{R}^2\right) \times H_{rad}^{1}\left(\mathbb{R}^2\right)} \leq \eps$ which is sufficiently small, then there exists a unique classical solution $\phi \in C^0\left(\left[0, T\right], H^s\left(\mathbb{R}^2\right)\right) \cap C^1\left(\left[0, T\right], H^{s-1}\left(\mathbb{R}^2\right)\right)$  with $s\ge 3$ for any $T>0$. 
\end{theorem}

\begin{theorem}\label{1.4}
		Consider Cauchy problem \eqref{equs} \eqref{ind} for  radially symmetrical solutions to the extremal hypersurface equations in (1+3)-dimensional Minkowski space-time. Under the assumption $\phi_{t}\big |_{r=0} =0$, when the initial data $\left(\phi_0, \phi_1\right) \in H_{rad}^2\left(\mathbb{R}^2\right) \times H_{rad}^{1} \left(\mathbb{R}^2\right)$ and $\|\left(\phi_0, \phi_1\right)\|_{H_{rad}^2\left(\mathbb{R}^2\right) \times H_{rad}^{1}\left(\mathbb{R}^2\right)} \leq \eps$ which is sufficiently small, then there exists a unique strong solution $\phi \in L^{\infty} \left(\left[0, T\right], H_{rad}^2\left(\mathbb{R}^2\right)\right) \cap W^{1,\infty} \left(\left[0, T\right], H_{rad}^{1}\left(\mathbb{R}^2\right)\right)$ for  any  $T>0$.
		Moreover, if we suppose $\phi$ and $\widetilde{\phi}$ are solutions of the extremal hypersurface equation \eqref{equ} evolved from initial data $\left(\phi_0, \phi_1\right)$ and $\left(\widetilde{\phi}_0, \widetilde{\phi}_1\right)$ which satisfy the smallness assumption, respectively, then we have the following estimates
		\begin{equation}
		\| \phi_{t}\left(t\right)-\widetilde{\phi}_t\left(t\right)\|_{L^2_{rad}} + \|\phi\left(t\right)-\widetilde{\phi}\left(t\right)\|_{H^1_{rad}} \leq 16 \left( \|\phi_0-\widetilde{\phi}_0\|_{H^1_{rad}} + \|\phi_1-\widetilde{\phi}_1\|_{L^2_{rad}}\right)
		\end{equation}
		for any $t\in \left(0, \infty\right)$.
\end{theorem}

\begin{theorem}\label{1.5}
	Consider Cauchy problem \eqref{equs} \eqref{ind} for  radially symmetrical solutions to the extremal hypersurface equations in (1+3)-dimensional Minkowski space-time. Under the assumption $\phi_{t}\big |_{r=0} =0$, for any initial data $\left(\phi_0, \phi_1\right) \in H_{rad}^2\left(\mathbb{R}^2\right) \times H_{rad}^{1} \left(\mathbb{R}^2\right)$,  then  there exists a positive number $\delta$ which depends on the $\|\left(\phi_0, \phi_1\right)\|_{H_{rad}^2\left(\mathbb{R}^2\right) \times H_{rad}^{1}\left(\mathbb{R}^2\right)}$ and the profiles of initial data $\left(\phi_0, \phi_1\right) $, and a unique strong solution $\phi \in L^{\infty} \left(\left[0, T^*\right], H_{rad}^2\left(\mathbb{R}^2\right)\right) \cap W^{1,\infty} \left(\left[0, T^*\right], H_{rad}^{1}\left(\mathbb{R}^2\right)\right)$ for any $T^* \leq \delta$.
\end{theorem}

\subsection{Outline of the article.} This article is organized as follows:
In section \ref{NP},  we introduce some notations and  preliminaries, which make more convenient to prove in the later section. In section \ref{bal}, we establish some balance law (i,e. $f_t + f_r = f^{re}$) associated with energy and ``momentum''. In section \ref{44}, we will derive a new div-curl lemma, which palys a key role to get the regularity estimates of solution.
In section \ref{55}, we note that ``div-curl'' structure is an ``inner product" in some sense, and calculate these ``inner product" induced by balance law. In section \ref{ei}, we use the new  div-curl lemma and continuous induction method to know the $\|\phi_t\|_{L^{\infty}_{t,x}} + 	\|\phi_r\|_{L^{\infty}_{t,x}}$ is in fact sufficient small, which guarantees the equation \eqref{equ} make sense. Moreover, based this fact, we can obtain the energy of the second derivative (critical energy) is smallness. In section \ref{77}, we use the new div-curl lemma to show the energy of the third derivative is uniformly bounded, which helps us to obatin the global existence for the classical solution to the equation \ref{equs}. In section \ref{88}, we use the  new div-curl lemma and homotopy method to prove the stability of solution for the system \eqref{equs}, which is useful to get the uniqueness of solution. In section \ref{99}, we combine the results in section \ref{ei}, section \ref{77}, and section \ref{88} to prove the theorem \ref{1.3}, theorem \ref{1.4}, and theorem \ref{1.5}.

\section{Notations and Preliminaries}\label{NP}
In this section, we will make some notations and give some preliminaries.

We mark $ A \lesssim B $ to mean there exists a constant $ C > 0 $ such that $ A \leqslant C B $. We indicate dependence on parameters via subscripts, e.g. $ A \lesssim_{x} B $ indicates $ A \leqslant CB $ for some $ C = C(x) > 0 $.

For convenience, let us define the following quantity:
\begin{equation*}
E_1\left(t\right):=E_1\left(\phi\left(t\right)\right):=\int \frac{\phi^2_{t}+\phi^2_{r}}{r^2}r dr,
\end{equation*}

\begin{equation*}
\hat{E}_1\left(t\right):=\hat{E}_1\left(\phi\left(t\right)\right):=\int\frac{1}{2}\frac{\phi^2_{r}+\phi^2_{t}}{r^2\Delta^{1/2}}r dr,
\end{equation*}

\begin{equation*}
E_2\left(t\right):=E_2\left(\phi\left(t\right)\right):=\int \left(\phi^2_{tt}+\phi^2_{tr}+\phi^2_{rr}\right) rdr,
\end{equation*}

\begin{equation*}
\hat{E}_2\left(t\right):=\hat{E}\left(\phi\left(t\right)\right):=\int\frac{1}{2}\frac{\phi^2_{tr}\left(1-\phi^2_{t}\right)+\phi^2_{tt}\left(1+\phi^2_{r}\right)}{\Delta^{3/2}} r dr,
\end{equation*}

\begin{equation*}
E_3\left(t\right):=E_3\left(\phi\left(t\right)\right):=\int \left(\phi^2_{ttt}+\phi^2_{ttr}+\phi^2_{trr}+ \frac{\phi^2_{tr}}{r^2} +\phi^2_{rrr}+\left( \frac{\phi_{rr}}{r}-\frac{\phi_{r}}{r^2}\right)^2  \right)r dr,
\end{equation*}

\begin{equation*}
E_{3,q}\left(t\right):=E_{3,q}\left(\phi\left(t\right)\right):= \int \left(\phi^2_{ttt} + \phi^2_{ttr}\right) rdr ,
\end{equation*}

\begin{equation*}
E_{3,s}\left(t\right):=E_{3,s}\left(\phi\left(t\right)\right):= \int \left(\phi^2_{trr}+ \frac{\phi^2_{tr}}{r^2}   \right)r dr,
\end{equation*}

\begin{equation*}
E_{3,l}\left(t\right):=E_{3,l}\left(\phi\left(t\right)\right):= \int \left(\phi^2_{rrr}+\left( \frac{\phi_{rr}}{r}-\frac{\phi_{r}}{r^2}\right)^2  \right)r dr,
\end{equation*}

\begin{equation*}
\begin{aligned}
\hat{E}_3\left(t\right):=\hat{E}_3\left(\phi\left(t\right)\right):&=\int \frac{1}{2}\frac{r\phi^2_{ttt}\left(1+\phi^2_{r}\right)
	+r\phi^2_{ttr}\left(1-\phi^2_{t}\right)}{\Delta^{3/2}} dr\\
&+\int r\phi_{ttr}\phi_{tr}\left(\frac{1-\phi^2_{t}}{\Delta^{3/2}}\right)_t +r\phi_{ttr}\phi_{tt}\left(\frac{\phi_{r}\phi_{t}}{\Delta^{3/2}}\right)_t dr,
\end{aligned}
\end{equation*}

\begin{equation*}
\tilde{E}_3\left(t\right):=\tilde{E}_3\left(\phi\left(t\right)\right):=\int \frac{1}{2}\frac{r\phi^2_{ttt}\left(1+\phi^2_{r}\right)+r\phi^2_{ttr}\left(1-\phi^2_{t}\right)}{\Delta^{3/2}} dr,
\end{equation*}

\begin{equation*}
M\left(t\right):=M\left(\phi\left(t\right)\right):=\int_{0}^{t}\int \frac{\phi^2_{t}+\phi^2_{r}}{r^2} dr ds,
\end{equation*}

\begin{equation*}
M_0\left(t\right):=M_0\left(\phi\left(t\right)\right):=\int_{0}^{t}\int \phi^2_{tt}\phi^2_{r}+\phi^2_{tr}\phi^2_{t}+\phi^2_{tt}\phi^2_{t}+\phi^2_{tr}\phi^2_{r} drds,
\end{equation*}

\begin{equation*}
M_h\left(t\right):=:M_h\left(\phi\left(t\right)\right):=\int_{0}^{t}\int \phi^2_{ttt}\phi^2_{r}+\phi^2_{ttr}\phi^2_{t}+\phi^2_{ttt}\phi^2_{t}+\phi^2_{ttr}\phi^2_{r} drds,
\end{equation*}

\begin{equation*}
M_{0,1}\left(t\right):=M_{0,1}\left(\phi\left(t\right)\right):=\int_{0}^{t}\int\phi^2_{tt}\phi^2_{t}+\phi^2_{tr}\phi^2_{t} drds,
\end{equation*}

\begin{equation*}
M_{0,2}\left(t\right):=M_{0,2}\left(\phi\left(t\right)\right):=\int_{0}^{t}\int \phi^2_{tt}\phi^2_{r} + \phi^2_{tr}\phi^2_{r} drds,
\end{equation*}

\begin{equation*}
M_{h,1}\left(t\right):=M_{h,1}\left(\phi\left(t\right)\right):=\int_{0}^{t}\int\phi^2_{ttt}\phi^2_{t}+\phi^2_{ttr}\phi^2_{t} drds,
\end{equation*}

\begin{equation*}
M_{h,2}\left(t\right):=M_{h,2}\left(\phi\left(t\right)\right):=\int_{0}^{t}\int \phi^2_{ttt}\phi^2_{r} + \phi^2_{ttr}\phi^2_{r} drds,
\end{equation*}

\begin{equation*}
M_1\left(t\right):=M_1\left(\phi\left(t\right)\right):=\int_{0}^{t}\int\frac{1}{4}\frac{\phi^2_{t}}{r^2\Delta^{1/2}} drds,
\end{equation*}

\begin{equation*}
M_2\left(t\right):=M_2\left(\phi\left(t\right)\right):=\int_{0}^{t}\int\frac{3}{4}\frac{\phi^2_{r}}{r^2\Delta^{1/2}}drds ,
\end{equation*}

\begin{equation*}
\eta_1\left(t\right):=\eta_1\left(\phi\left(t\right)\right):=\int^t_0\int det A_m\left(\phi\left(t\right)\right)
drds,
\end{equation*}

\begin{equation*}
\eta_2\left(t\right):=\eta_2\left(\phi\left(t\right)\right):=\int^t_0\int det A\left(\phi\left(t\right)\right)
drds,
\end{equation*}

\begin{equation*}
\xi_1\left(t\right):=\xi_1\left(\phi\left(t\right)\right):=\int^t_0\int det B_m\left(\phi\left(t\right)\right) drds,
\end{equation*}

\begin{equation*}
\xi_2\left(t\right):=\xi_2\left(\phi\left(t\right)\right):=\int^t_0\int det B\left(\phi\left(t\right)\right) drds,
\end{equation*}

\begin{equation*}
\zeta_1\left(t\right):=\zeta_1\left(\phi\left(t\right)\right):=\int^t_0\int det C_m\left(\phi\left(t\right)\right) drds,
\end{equation*}

\begin{equation*}
\zeta_2\left(t\right):=\zeta_2\left(\phi\left(t\right)\right):=\int^t_0\int det C\left(\phi\left(t\right)\right) drds,
\end{equation*}

\begin{equation*}
\gamma_1\left(t\right):=\gamma_1\left(\phi\left(t\right)\right):=\int^t_0\int det D_m\left(\phi\left(t\right)\right) drds,
\end{equation*}

\begin{equation*}
\gamma_2\left(t\right):=\gamma_2\left(\phi\left(t\right)\right):=\int^t_0\int det D\left(\phi\left(t\right)\right) drds,
\end{equation*}
where the $det A_m\left(\phi\left(t\right)\right)$, $det A\left(\phi\left(t\right)\right)$, $det B_m\left(\phi\left(t\right)\right)$, $det B\left(\phi\left(t\right)\right)$, $det C_m\left(\phi\left(t\right)\right)$, $det C\left(\phi\left(t\right)\right)$, $det D_m\left(\phi\left(t\right)\right)$, $det D\left(\phi\left(t\right)\right)$ are defined in section
\ref{55}.

\section{law of equilibrium}\label{bal}

In this section, We're going to  multiply both sides of the original equation \eqref{equ} or the derivative (second derivative) of the original equation \eqref{1-der}, \eqref{1'-der}, \eqref{2-der}  by the corresponding multipliers to get some equalities that associate with energy and ``momentum", and then write those equalities in terms of the balance law, i,e. $f_t + f_r = f^{re}$.

Let us multiply both sides of \eqref{equ} by $\frac{\phi_{t}}{r^2}$, we can obtain:
\begin{equation}
\frac{\phi_{t}}{r^2}\left[\left(\frac{r\phi_{t}}{\Delta^{1/2}}\right)_t-\left(\frac{r\phi_{r}}{\Delta^{1/2}}\right)_r\right]=0.
\end{equation}

Then we know:
\begin{equation}\label{ph1}
	\frac{1}{2}\left(\frac{\phi^2_{r}+\phi^2_{t}}{r\Delta^{1/2}}\right)_t-\left(\frac{\phi_{t}\phi_{r}}{r\Delta^{1/2}}\right)_r=Q^1,
\end{equation}
where
\begin{equation*}
	Q^1=-\frac{1}{2}\phi^2_{t}\left(\frac{1}{r\Delta^{1/2}}\right)_t+\frac{1}{2}\phi^2_r\left(\frac{1}{r\Delta^{1/2}}\right)_t+\frac{2\phi_{t}\phi_{r}}{r^2\Delta^{1/2}}.
\end{equation*}

Similarly, we multiply both sides of \eqref{equ} by $\frac{\phi_{r}}{r^2}$, we have:
\begin{equation}
\frac{\phi_{r}}{r^2}\left[\left(\frac{r\phi_{t}}{\Delta^{1/2}}\right)_t-\left(\frac{r\phi_{r}}{\Delta^{1/2}}\right)_r\right]=0.
\end{equation}

It is not difficult to obtain:
\begin{equation}\label{ph2}
\left(\frac{\phi_r \phi_{t}}{r\Delta^{1/2}}\right)_t -\frac{1}{2} \left(\frac{\phi^2_{t}+\phi^2_{r}}{r\Delta^{1/2}}\right)_r=Q^2,
\end{equation}
where 
\begin{equation*}
Q^2=-\frac{1}{2}\phi^2_{t}\left(\frac{1}{r\Delta^{1/2}}\right)_r+\frac{1}{2}\phi^2_r\left(\frac{1}{r\Delta^{1/2}}\right)_r+\frac{2\phi^2_{r}}{r^2\Delta^{1/2}}.
\end{equation*}

Taking the first derivative of the equation \eqref{equ} with respect to time, we can get:
\begin{equation}\label{1-der}
\left(\frac{r \phi_{tt}\left(1+\phi^2_r\right)-r \phi_t\phi_r\phi_{tr}}{\Delta^{3/2} }\right)_t - \left(\frac{r \phi_{tr}\left( 1-\phi^2_t\right) + r \phi_r\phi_t\phi_{tt}}{\Delta^{3/2}}\right)_r = 0 .
\end{equation}

We can rewrite \eqref{1-der} as:
\begin{equation}\label{1-derr}
	r\left(\frac{\phi_{tt}\left(1+\phi^2_{r}\right)-\phi_{t}\phi_{r}\phi_{tr}}{\Delta^{3/2}}\right)_t-r\left(\frac{\phi_{tr}\left(1-\phi^2_{t}\right)+\phi_{r}\phi_{t}\phi_{tt}}{\Delta^{3/2}}\right)_r-\left(\frac{\phi_{r}}{\Delta^{1/2}}\right)_t=0.
\end{equation}

Similarly, taking the the first derivative of the equation \eqref{equ} with respect to space, we can get:
\begin{equation}\label{1'-der}
\left(\frac{r \phi_{tr}\left(1+\phi^2_r\right)-r \phi_t\phi_r\phi_{rr} +\phi_{t}\Delta}{\Delta^{3/2} }\right)_t - \left(\frac{r \phi_{rr}\left( 1-\phi^2_t\right) + r \phi_r\phi_t\phi_{tr}+\phi_{r}\Delta}{\Delta^{3/2}}\right)_r = 0 .
\end{equation}

We can rewrite \eqref{1'-der} as:
\begin{equation}\label{1'-derr}
	r\left(\frac{\phi_{tr}\left(1+\phi^2_{r}\right)-\phi_{t}\phi_{r}\phi_{rr}}{\Delta^{3/2}}\right)_t-r\left(\frac{\phi_{rr}\left(1-\phi^2_{t}\right)+\phi_{r}\phi_{t}\phi_{tr}}{\Delta^{3/2}}\right)_r-\left(\frac{\phi_{r}}{\Delta^{1/2}}\right)_r+\frac{\phi_{r}}{r\Delta^{1/2}}=0.
\end{equation}

Let us multiply both sides of \eqref{1-der} by $\phi_{tt}$, we can obtain:
\begin{equation}
    \phi_{tt}\left[\left( \frac{r \phi_{tt}\left(1+\phi^2_r\right)-r \phi_t\phi_r\phi_{tr}}{\Delta^{3/2} }\right)_t - \left(\frac{r \phi_{tr}\left( 1-\phi^2_t\right) + r \phi_r\phi_t\phi_{tt}}{\Delta^{3/2}}\right)_r\right] = 0 .
\end{equation}

By some  simple calculation, it is not difficult to see the following equality:
\begin{equation}\label{ph3}
\begin{aligned}
	\frac{1}{2} \left(\frac{r\phi^2_{tt}\left(1+\phi^2_r\right)+r\phi^2_{tr}\left(1-\phi^2_t\right)}{\Delta^{3/2}}\right)_t-\left( \frac{r\phi_{tt}\phi_{tr}\left(1-\phi^2_t\right)+r\phi_r\phi_t\phi^2_{tt}}{\Delta^{3/2}}\right)_r=W^1,
\end{aligned}
\end{equation}
where
\begin{equation*}
	W^1=-\frac{1}{2}\phi^2_{tt}\left(\frac{r\left(1+\phi^2_r\right)}{\Delta^{3/2}}\right)_t+\phi_{tt}\phi_{tr}\left(\frac{r\phi_t\phi_r}{\Delta^{3/2}}\right)_t+\frac{1}{2}\phi^2_{tr}\left(\frac{r\left(1-\phi^2_t\right)}{\Delta^{3/2}}\right)_t.
\end{equation*}

In fact, 
\begin{equation}
\begin{aligned}
W^1&=\left(\frac{r}{\Delta^{3/2}}\right)_t\frac{1}{2}\left(\phi^2_{tr}\left(1-\phi^2_t\right)+2\phi_{tt}\phi_{tr}\phi_t\phi_r-\phi^2_{tt}\left(1+\phi^2_r\right)\right)\\
&+\left(\frac{r}{\Delta^{3/2}} \right)\left(-\phi_r\phi_{tr}\phi^2_{tt}+\phi_{tt}\phi_{tr}\left(\phi_t\phi_r\right)_t -\phi^2
_{tr}\phi_t\phi_{tt}\right)\\
&=\left(\frac{r}{\Delta^{3/2}}\right)_t\left(\phi^2_{tr}\left(1-\phi^2_t\right)+2\phi_{tt}\phi_{tr}\phi_t\phi_r-\phi^2_{tt}\left(1+\phi^2_r\right)\right) .
\end{aligned}
\end{equation}

Let us multiply both sides of \eqref{1-derr} by $\left(\phi_{tr}+ \frac{\phi_t}{2r}\right)$, we can obtain:
\begin{equation}
	\begin{aligned}
	&\left(\phi_{tr}+ \frac{\phi_t}{2r}\right)\left[r\left(\frac{\phi_{tt}\left(1+\phi^2_{r}\right)-\phi_{t}\phi_{r}\phi_{tr}}{\Delta^{3/2}}\right)_t-r\left(\frac{\phi_{tr}\left(1-\phi^2_{t}\right)+\phi_{r}\phi_{t}\phi_{tt}}{\Delta^{3/2}}\right)_r\right]\\
	&-\left(\phi_{tr}+\frac{\phi_{t}}{2r}\right)\left(\frac{\phi_{r}}{\Delta^{1/2}}\right)_t=0.
	\end{aligned}
\end{equation}

According to a simple calculation, it is not difficult to see following equality:
\begin{equation}\label{ph5}
	\begin{aligned}
	&\left(\frac{r\phi_{tt}\phi_{tr}\left(1+\phi^2_{r}\right)-r\phi_{t}\phi_{r}\phi^2_{tr}}{\Delta^{3/2}}+\frac{1}{2}\frac{\phi_{t}\phi_{tt}\left(1+\phi^2_{r}\right)-\phi^2_{t}\phi_{r}\phi_{tr}}{\Delta^{3/2}}\right)_t\\
	&-\left(\frac{1}{2}\frac{r\phi^2_{tr}\left(1-\phi^2_{t}\right)+r\phi^2_{tt}\left(1+\phi^2_{r}\right)}{\Delta^{3/2}}+\frac{1}{2}\frac{\phi_{t}\phi_{tr}\left(1-\phi^2_{t}\right)+\phi^2_{t}\phi_{r}\phi_{tt}}{\Delta^{3/2}}+\frac{1}{4}\frac{\phi^2_{t}}{r\Delta^{1/2}}\right)_r\\
	&=P^1,
	\end{aligned}
\end{equation}
where
\begin{equation}
\begin{aligned}
	&P^1=r\left(\frac{1}{\Delta^{3/2}}\right)_r\left(-\frac{\phi^2_{tt}}{2}\left(1+\phi^2_{r}\right)+\phi_{tr}\phi_{tt}\phi_{r}\phi_{t}+\frac{\phi^2_{tr}}{2}\left(1-\phi^2_{t}\right)\right)\\
	&+\left(\frac{r}{\Delta^{3/2}}\right)\left(\phi_{tt}\phi_{r}-\phi_{tr}\phi_{t}\right)\left(\phi^2_{tr}-\phi_{tt}\phi_{rr}\right)\\
	&+\frac{1}{4}\frac{\phi^2_{t}}{r^2\Delta^{1/2}}-\left(\frac{1}{\Delta^{1/2}}\right)_r\frac{\phi^2_{t}}{4r}+\left(\frac{1}{\Delta^{1/2}}\right)_t\frac{\phi_{t}\phi_{r}}{2r}.
\end{aligned}
\end{equation}

Similarly, we multiply both sides of \eqref{1'-derr} by$\left(\phi_{rr}+\frac{\phi_{r}}{2r}\right)$, we have:
\begin{equation}
\begin{aligned}
&\left(\phi_{rr}+ \frac{\phi_r}{2r}\right)\left[r\left(\frac{\phi_{tr}\left(1+\phi^2_{r}\right)-\phi_{t}\phi_{r}\phi_{rr}}{\Delta^{3/2}}\right)_t-r\left(\frac{\phi_{rr}\left(1-\phi^2_{t}\right)+\phi_{r}\phi_{t}\phi_{tr}}{\Delta^{3/2}}\right)_r\right]\\
&-\left(\phi_{rr}+\frac{\phi_{r}}{2r}\right)\left(\frac{\phi_{r}}{\Delta^{1/2}}\right)_r+\left(\phi_{rr}+\frac{\phi_{r}}{2r}\right)\left(\frac{\phi_{r}}{r\Delta^{1/2}}\right)=0.
\end{aligned}
\end{equation}

By some simple calculation, we can get:
\begin{equation}\label{ph6}
\begin{aligned}
&\left(\frac{r\phi_{rr}\phi_{tr}\left(1+\phi^2_{r}\right)-r\phi_{t}\phi_{r}\phi^2_{rr}}{\Delta^{3/2}}+\frac{1}{2}\frac{\phi_{r}\phi_{tr}\left(1+\phi^2_{r}\right)-\phi_{t}\phi^2_{r}\phi_{rr}}{\Delta^{3/2}}\right)_t\\
&-\left(\frac{1}{2}\frac{r\phi^2_{rr}\left(1-\phi^2_{t}\right)+r\phi^2_{tr}\left(1+\phi^2_{r}\right)}{\Delta^{3/2}}+\frac{1}{2}\frac{\phi_{r}\phi_{rr}\left(1-\phi^2_{t}\right)+\phi_{t}\phi^2_{r}\phi_{tr}}{\Delta^{3/2}}+\frac{1}{4}\frac{\phi^2_{r}}{r\Delta^{1/2}}\right)_r\\
&=P^2,
\end{aligned}
\end{equation}
where
\begin{equation}
	\begin{aligned}
	&P^2= r\left(\frac{1}{\Delta^{3/2}}\right)_r\left(-\frac{\phi^2_{tr}}{2}\left(1+\phi^2_{r}\right)+\phi_{tr}\phi_{rr}\phi_{r}\phi_{t}+\frac{\phi^2_{rr}}{2}\left(1-\phi^2_{t}\right)\right)\\
	&+\frac{3}{4}\frac{\phi^2_{r}}{r^2\Delta^{1/2}}+\left(\frac{1}{\Delta^{1/2}}\right)_r\frac{3\phi^2_{r}}{4r}.
	\end{aligned}
\end{equation}

Taking the second derivative of the equation \eqref{equ} with respect to time, we can get:
\begin{equation}\label{2-der}
\left(\frac{r \phi_{tt}\left(1+\phi^2_r\right)-r \phi_t\phi_r\phi_{tr}}{\Delta^{3/2} }\right)_{tt} - \left(\frac{r \phi_{tr}\left( 1-\phi^2_t\right) + r \phi_r\phi_t\phi_{tt}}{\Delta^{3/2}}\right)_{tr} = 0 .
\end{equation}

We multiply both sides of \eqref{2-der} by $\phi_{ttt}$, we have:
\begin{equation}
\phi_{ttt} \left[\left(\frac{r \phi_{tt}\left(1+\phi^2_r\right)-r \phi_t\phi_r\phi_{tr}}{\Delta^{3/2} }\right)_{tt} - \left(\frac{r \phi_{tr}\left( 1-\phi^2_t\right) + r \phi_r\phi_t\phi_{tt}}{\Delta^{3/2}}\right)_{tr} \right]=0.
\end{equation}

By some  simple calculation, it is not difficult to see the following equality:
\begin{equation}\label{ph7}
	\begin{aligned}
	&\left(\frac{1}{2}\frac{r\phi^2_{ttt}\left(1+\phi^2_{r}\right)+r\phi^2_{ttr}\left(1-\phi^2_{t}\right)}{\Delta^{3/2}}+r\phi_{ttr}\phi_{tr}\left(\frac{1-\phi^2_{t}}{\Delta^{3/2}}\right)_t +r\phi_{ttr}\phi_{tt}\left(\frac{\phi_{r}\phi_{t}}{\Delta^{3/2}}\right)_t \right)_t\\
	&-\left(\frac{r\phi_{ttt}\phi_{ttr}\left(1-\phi^2_{t}\right)+r\phi^2_{ttt}\phi_{t}\phi_{r}}{\Delta^{3/2}} + r\phi_{ttt}\phi_{tr}\left(\frac{1-\phi^2_{t}}{\Delta^{3/2}}\right)_t+r\phi_{ttt}\phi_{tt}\left(\frac{\phi_{r}\phi_{t}}{\Delta^{3/2}}\right)_t\right)_r\\
	&= T^1,
	\end{aligned}
\end{equation}
where
\begin{equation*}
	\begin{aligned}
	T^1&=r\left(\frac{1}{\Delta^{3/2}}\right)_{tt}\left(-\phi_{ttt}\phi_{tt}\left(1+\phi^2_{r}\right)+\phi_{ttt}\phi_{tr}\phi_{t}\phi_{r}+\phi_{ttr}\phi_{tt}\phi_{r}\phi_{t}+\phi_{ttr}\phi_{tr}\left(1-\phi^2_{t}\right)\right)\\
	&+r\left(\frac{1}{\Delta^{3/2}}\right)_t\frac{3}{2}\left(-\phi^2_{ttt}\left(1+\phi^2_{r}\right)+2\phi_{ttt}\phi_{ttr}\phi_{t}\phi_{r}+\phi^2_{ttr}\left(1-\phi^2_{t}\right)\right)\\
	&+r\left(\frac{1}{\Delta^{3/2}}\right)_t2\left(\phi_{ttt}\phi_{tr}-\phi_{ttr}\phi_{tt}\right)\left(\phi_{tr}\phi_{t}-\phi_{tt}\phi_{r}\right)\\
	&+r\left(\frac{1}{\Delta^{3/2}}\right)2\left(\phi_{ttr}\phi_{t}-\phi_{ttt}\phi_{r}\right)\left(\phi_{ttt}\phi_{tr}-\phi_{ttr}\phi_{tt}\right).
	\end{aligned}
\end{equation*}

In fact,
\begin{equation}
	\begin{aligned}
	T^1&=r\left(\frac{1}{\Delta^{3/2}}\right)_{tt}\left(-\phi_{ttt}\phi_{tt}\left(1+\phi^2_{r}\right)+\phi_{ttt}\phi_{tr}\phi_{t}\phi_{r}+\phi_{ttr}\phi_{tt}\phi_{r}\phi_{t}+\phi_{ttr}\phi_{tr}\left(1-\phi^2_{t}\right)\right)\\
	&+r\frac{9}{2}\frac{1}{\Delta^{5/2}}\left(\phi_{r}\phi_{ttr}-\phi_{t}\phi_{ttt}\right)\left(\phi_{tr}\phi_{ttr}\left(1-\phi^2_{t}\right)-\phi_{tt}\phi_{ttt}\left(1+\phi^2_{r}\right)+2\phi_{ttt}\phi_{tr}\phi_{r}\phi_{t}\right)\\
	&+r\frac{9}{2}\frac{1}{\Delta^{5/2}}\phi_{t}\phi_{ttr}\left(1-\phi^2_{t}\right)\left(\phi_{ttt}\phi_{tr}-\phi_{ttr}\phi_{tt}\right)\\
	&+r\frac{9}{2}\frac{1}{\Delta^{5/2}}\phi_{r}\phi_{ttt}\left(1+\phi^2_{r}\right)\left(\phi_{ttr}\phi_{tt}-\phi_{ttt}\phi_{tr}\right)\\
	&+2r\frac{9}{2}\frac{1}{\Delta^{5/2}}
	\phi^2_{t}\phi_{r}\phi_{ttt}\left(\phi_{ttt}\phi_{tr}-\phi_{ttr}\phi_{tt}\right)\\
	&+r\left(\frac{1}{\Delta^{3/2}}\right)_t2\left(\phi_{ttt}\phi_{tr}-\phi_{ttr}\phi_{tt}\right)\left(\phi_{tr}\phi_{t}-\phi_{tt}\phi_{r}\right)\\
	&+r\left(\frac{1}{\Delta^{3/2}}\right)2\left(\phi_{ttr}\phi_{t}-\phi_{ttt}\phi_{r}\right)\left(\phi_{ttt}\phi_{tr}-\phi_{ttr}\phi_{tt}\right).
	\end{aligned}
\end{equation}

	


\section{New div-curl type lemma}\label{44}
In this section, we will derive a new div-curl lemma, which is useful for us to get the regularity estimates of solution.

\begin{lemma}\label{dc}
	Suppose that
	\begin{equation}
	\left\{
	\begin{aligned}
	&f_t^{11} + f_r^{12} =G^1\\
	& f_t^{21}-f_r^{22}=G^2
	\end{aligned}
	\right.
	\end{equation}
	
	\begin{equation}
	\begin{aligned}
	f^ {12} \rightarrow 0, r\rightarrow 0.
	\end{aligned}
	\end{equation}
	Then there hold
	\begin{equation}
	\begin{aligned}
	\int_{0}^{T}\int_{0}^{\infty} f^{11}f^{22}+f^{12}f^{21} &\lesssim
	\left(\|f^{11}\left(0\right)\|_{L^1} + \sup\limits_{0\leq t\leq T}  \|f^{11}\left(t\right)\|_{L^1} + \int_{0}^{T}\int_{0}^{\infty}\lvert G^1\rvert \right)\\
	&\cdot\left(\|f^{21}\left(0\right)\|_{L^1} +\sup\limits_{0\leq t\leq T} \|f^{21}\left(t\right)\|_{L^1}  +\int_{0}^{T}\int_{0}^{\infty}\lvert G^2\rvert \right)
	\end{aligned}
	\end{equation}
	provided that the right side is bounded.
\end{lemma}

\begin{proof}
	We note
	\begin{equation*}
	\left(\int_{0}^{r} f^{11} \right)_t +f^{12}=\int_{0}^{r} G^1.
	\end{equation*}
	
	Then 
	\begin{equation}\label{dd1}
	f^{21}\int_{0}^{r} f_t^{11}+f^{12}f^{21}=\int_{0}^{r}G^1f^{21},
	\end{equation}
	
	\begin{equation}\label{dd2}
	f^{21}_t\int_{0}^{r} f^{11}-f_r^{22}\int_{0}^{r}f^{11}= G^2\int_{0}^{r}f^{11}.
	\end{equation}
    
    \eqref{dd1}+\eqref{dd2}:
	\begin{equation}
	\int_{0}^{\infty} \left(\int_{0}^{r} f^{11}f^{21}\right)_t+\int_{0}^{\infty} 
	\left(f^{12}f^{21}-f_r^{22}\int_{0}^{r}f^{11}\right)=\int_{0}^{\infty}\left(f^{21}\int_{0}^{r}G^1 + G^2\int_{0}^{r}f^{11}\right).
	\end{equation}
	
	We have:
	\begin{equation}
	\begin{aligned}
	\int_{0}^{T}\int_{0}^{\infty} f^{11}f^{22}+f^{12}f^{21}=& \int_{0}^{\infty} \left(\int_{0}^{r} f^{11}f^{21}\right)\left(0\right)- \left(\int_{0}^{r} f^{11}f^{21}\right)\left(T\right) \\
	&+\int_{0}^{T}\int_{0}^{\infty}\left(f^{21}\int_{0}^{r}G^1 + G^2\int_{0}^{r}f^{11}\right)\\
	&:= \mathcal{A}_1+\mathcal{A}_2+\mathcal{A}_3,
	\end{aligned}
	\end{equation}
where	
	\begin{equation*}
	\begin{aligned}
	\lvert \mathcal{A}_1\rvert&\lesssim \|f^{11}\left(0\right)\|_{L^1}\|f^{21}\left(0\right)\|_{L^1} + \|f^{11}\left(T\right)\|_{L^1}\|f^{21}\left(T\right)\|_{L^1},\\
	\end{aligned}
	\end{equation*}
	
	\begin{equation*}
	\begin{aligned}
	\lvert \mathcal{A}_2\rvert&\lesssim\int_{0}^{T} \|f^{21}\left(t\right)\|_{L^1}\|G^1\left(t\right)\|_{L^1}\\
	\lesssim&\sup\limits_{0\leq t\leq T} \| f^{21}\left(t\right)\|_{L^1}\left(\int_{0}^{T}\int_{0}^{\infty}\lvert G^1\rvert\right),
	\end{aligned}
	\end{equation*} 
	
	\begin{equation*}
	\begin{aligned}
	\lvert \mathcal{A}_3\rvert&\lesssim\int_{0}^{T} \|f^{11}\left(t\right)\|_{L^1}\|G^2\left(t\right)\|_{L^1}\\
	\lesssim&\sup\limits_{0\leq t\leq T} \| f^{11}\left(t\right)\|_{L^1}\left(\int_{0}^{T}\int_{0}^{\infty}\lvert G^2\rvert\right).
	\end{aligned}
	\end{equation*}

	Based analysis above, we complete the proof. 
\end{proof}

\begin{remark}\label{neiji}
	In fact, the quantity $f^{11}f^{22}+f^{12}f^{21}$ is an ``inner product" in some sense. But by coincidence, this quantity equals to the determinant of a matrix.  We  can represent the matrix in the following form
	
	\begin{equation*}
	\mathbb{A}=
	\begin{pmatrix}
f^{11} & -f^{12} \\
	f^{21}           & f^{22}
	\end{pmatrix}.
	\end{equation*}
	
\end{remark}

\section{Some calculus from law of equilibrium}\label{55}
In this section, we will calculate some ``inner product''   induced by balance law, which paly a crucial role in the  estimates of regularity in the later sections. 

From the above remark \ref{neiji}, we can calculate some determinant of matrix to obtain some ``inner product'' .

We consider the ``inner product" induced by the law of equilibrium \eqref{ph1} and \eqref{ph5}.
\begin{equation}
A\left(\phi\left(t\right)\right):=
\begin{pmatrix}
a &b \\
c^1+c^2           & d^1+d^2+d^3 
\end{pmatrix},
\end{equation}
where
\begin{equation*}
	a:=\frac{1}{2} \frac{\phi^2_{t}+\phi^2_{r}}{r\Delta^{1/2}},
\end{equation*}

\begin{equation*}
	b:=\frac{\phi_{t}\phi_{r}}{r\Delta^{1/2}},
\end{equation*}

\begin{equation*}
	c^1:=\frac{r\phi_{tt}\phi_{tr}\left(1+\phi^2_{r}\right)-r\phi_{t}\phi_{r}\phi^2_{tr}}{\Delta^{3/2}},
\end{equation*}

\begin{equation*}
	c^2:=\frac{1}{2}\frac{\phi_{t}\phi_{tt}\left(1+\phi^2_{r}\right)-\phi^2_{t}\phi_{r}\phi_{tr}}{\Delta^{3/2}},
\end{equation*}

\begin{equation*}
	d^1:=\frac{1}{2}\frac{r\phi^2_{tr}\left(1-\phi^2_{t}\right)+r\phi^2_{tt}\left(1+\phi^2_{r}\right)}{\Delta^{3/2}},
\end{equation*}

\begin{equation*}
	d^2:=\frac{1}{2}\frac{\phi_{t}\phi_{tr}\left(1-\phi^2_{t}\right)+\phi^2_{t}\phi_{r}\phi_{tt}}{\Delta^{3/2}},
\end{equation*}

\begin{equation*}
	d^3:=\frac{1}{4}\frac{\phi^2_{t}}{r\Delta^{1/2}}.
\end{equation*}


By some simple calcaution, we have:
\begin{equation}
\begin{aligned}
det A\left(\phi\left(t\right)\right)&=ad^1+ad^2+ad^3-bc^1-bc^2\\
&=ad^1-bc^1+ad^2-bc^2+ad^3\\
&=det A_m\left(\phi\left(t\right)\right)\\
&+\frac{1}{4r\Delta^2} \left(\phi^2_{t}+\phi^2_{r}\right)\left(\phi_{t}\phi_{tr}\left(1-\phi^2_{t}\right)+\phi^2_{t}\phi_{r}\phi_{tt}\right)\\
&-\frac{1}{2r\Delta^2}\left(\phi_{t}\phi_{r}\right)\left(\phi_{t}\phi_{tr}\left(1-\phi^2_{t}\right)+\phi^2_{t}\phi_r\phi_{tt}\right)\\
&+\frac{1}{8\Delta}\left(\phi^2_{t}+\phi^2_{r}\right)\frac{\phi^2_t}{r^2},
\end{aligned}
\end{equation}
where
\begin{equation}
\begin{aligned}
det A_m\left(\phi\left(t\right)\right) &:=\frac{1}{4\Delta^2}\left(\phi_{t}\phi_{tt}-\phi_{r}\phi_{tr}\right)^2+\frac{1}{4\Delta^2}\left(\phi_{t}\phi_{tr}-\phi_{r}\phi_{tt}\right)^2\\
&+\frac{1}{4\Delta^2}\left(\phi_{tt}\phi_{r}-\phi_{t}\phi_{tr}\right)^2\left(\phi^2_{t}+\phi^2_{r}\right)\\
&+\frac{1}{2\Delta^2}\phi_{r}\phi_{t}\left(\phi_{tt}\phi_{r}-\phi_{tr}\phi_{t}\right)\left(\phi_{tt}\phi_{t}-\phi_{tr}\phi_{r}\right)\\
&+\frac{1}{2\Delta^2}\left(\phi^2_{t}+\phi^2_{r}\right)\left(\phi_{tt}\phi_{r}-\phi_{tr}\phi_{t}\right)\left(\phi_{tr}\phi_{r}-\phi_{tt}\phi_{t}\right)\\
&-\frac{1}{2\Delta^2}\left(\phi^2_{t}+\phi^2_{r}\right)\left(\phi_{tt}\phi_{r}-\phi_{t}\phi_{tr}\right)^2 .\\
\end{aligned}
\end{equation}

We consider the ``inner product" induced by the law of equilibrium \eqref{ph3} and \eqref{ph5}.
\begin{equation}
B\left(\phi\left(t\right)\right):=
\begin{pmatrix}
a &b \\
c^1+c^2           & d^1+d^2+d^3 

\end{pmatrix},
\end{equation}
where 
\begin{equation*}
	a:=\frac{1}{2}\frac{r\phi^2_{tt}\left(1+\phi^2_r\right)+r\phi^2_{tr}\left(1-\phi^2_t\right)}{\Delta^{3/2}},
\end{equation*}
\begin{equation*}
	b:= \frac{r\phi_{tr}\phi_{tt}\left(1-\phi^2_t\right)+r\phi_r\phi_t\phi^2_{tt}}{\Delta^{3/2}},
\end{equation*}
\begin{equation*}
c^1:=\frac{r\phi_{tt}\phi_{tr}\left(1+\phi^2_{r}\right)-r\phi_{t}\phi_{r}\phi^2_{tr}}{\Delta^{3/2}},
\end{equation*}
\begin{equation*}
c^2:=\frac{1}{2}\frac{\phi_{t}\phi_{tt}\left(1+\phi^2_{r}\right)-\phi^2_{t}\phi_{r}\phi_{tr}}{\Delta^{3/2}},
\end{equation*}
\begin{equation*}
d^1:=\frac{1}{2}\frac{r\phi^2_{tr}\left(1-\phi^2_{t}\right)+r\phi^2_{tt}\left(1+\phi^2_{r}\right)}{\Delta^{3/2}},
\end{equation*}
\begin{equation*}
d^2:=\frac{1}{2}\frac{\phi_{t}\phi_{tr}\left(1-\phi^2_{t}\right)+\phi^2_{t}\phi_{r}\phi_{tt}}{\Delta^{3/2}},
\end{equation*}
\begin{equation*}
d^3:=\frac{1}{4}\frac{\phi^2_{t}}{r\Delta^{1/2}}.
\end{equation*}


By some simple calcaution, we have:
\begin{equation}\label{dc3}
\begin{aligned}
det B\left(\phi\left(t\right)\right) &=ad^1+ad^2+ad^3-bc^1-bc^2\\
&=ad^1-bc^1+ad^2-bc^2+ad^3\\
&=det B_m\left(\phi\left(t\right)\right)\\
&+\frac{1}{8\Delta^2}\left(\phi^2_{t}\phi^2_{tt}\left(1+\phi^2_{r}\right)+\phi^2_{t}\phi^2_{tr}\left(1-\phi^2_t\right)\right)\\
&+\frac{r}{4\Delta^3}\phi^2_{t}\phi_{r}\phi_{tt}\left(\left(1-\phi^2_{t}\right)\phi^2_{tr}-\left(1+\phi^2_{r}\right)\phi^2_{tt}+2\phi_t\phi_{r}\phi_{tt}\phi_{tr}\right)\\
&-\frac{r}{4\Delta^3}\phi^2_{t}\frac{\phi_{t}\phi_{r}\phi_{tr}\phi_{tt}\Delta}{r}\\
&-\frac{r}{\Delta^3}\left(1-\phi^2_{t}\right)\phi_{t}\phi_{tr}\left(\phi^2_{tt}\left(1+\phi^2_{r}\right)-2\phi_{t}\phi_{r}\phi^2_{tt}\phi^2_{tr}-\phi^2_{tr}\left(1-\phi^2_{t}\right)\right),
\end{aligned}
\end{equation}
where
\begin{equation*}
	det B_m \left(\phi\left(t\right)\right):=r^2 \left(\frac{\phi^2_{tr}\left(1-\phi^2_t\right)-\phi^2_{tt}\left(1+\phi^2_r\right)}{2} + \phi_t\phi_r\phi_{tt}\phi_{tr}\right)^2.
\end{equation*}

We consider the ``inner product" induced by the law of equilibrium \eqref{ph2} and \eqref{ph7}.
\begin{equation}
C\left(\phi\left(t\right)\right):=
\begin{pmatrix}
a &b \\
c^1+c^2+c^3           & d^1+d^2+d^3 

\end{pmatrix},
\end{equation}
where
\begin{equation*}
	a:= \frac{\phi_{t}\phi_{r}}{r\Delta^{1/2}},
\end{equation*}

\begin{equation*}
	b:=\frac{1}{2} \frac{\phi^2_{t}+\phi^2_{r}}{r\Delta^{1/2}},
\end{equation*}

\begin{equation*}
   c^1:= \frac{1}{2}\frac{r\phi^2_{ttt}\left(1+\phi^2_{r}\right)+r\phi^2_{ttr}\left(1-\phi^2_{t}\right)}{\Delta^{3/2}},
\end{equation*}

\begin{equation*}
c^2:=r\phi_{ttr}\phi_{tr}\left(\frac{1-\phi^2_{t}}{\Delta^{3/2}}\right)_t,
\end{equation*}

\begin{equation*}
	c^3:= r\phi_{ttr}\phi_{tt}\left(\frac{\phi_{r}\phi_{t}}{\Delta^{3/2}}\right)_t ,
\end{equation*}

\begin{equation*}
d^1:= \frac{r\phi_{ttt}\phi_{ttr}\left(1-\phi^2_{t}\right)+r\phi^2_{ttt}\phi_{t}\phi_{r}}{\Delta^{3/2}},
\end{equation*}

\begin{equation*}
	d^2:=r\phi_{ttt}\phi_{tr}\left(\frac{1-\phi^2_{t}}{\Delta^{3/2}}\right)_t,
\end{equation*}

\begin{equation*}
	d^3:=r\phi_{ttt}\phi_{tt}\left(\frac{\phi_{r}\phi_{t}}{\Delta^{3/2}}\right)_t.
\end{equation*}

By some simple calcaution, we have:
\begin{equation}
\begin{aligned}
det C\left(\phi\left(t\right)\right)&=ad^1+ad^2+ad^3-bc^1-bc^2-b^3\\
&=ad^1-bc^1+ad^2-bc^2+ad^3-bc^3\\
&=det C_m\left(\phi\left(t\right)\right)
+\frac{\phi_{t}\phi_{r}}{\Delta^2}\left(\phi_{ttt}\phi_{tt}\phi_{tt}\phi_{r}-\phi_{ttt}\phi_{tr}\phi_{tt}\phi_{t}\right)\\
&-\frac{\phi^2_{t}+\phi^2_{r}}{2\Delta^2}\left(\phi_{ttr}\phi_{tt}\phi_{tt}\phi_{r}-\phi_{ttr}\phi_{tr}\phi_{tt}\phi_{t}\right)\\
&+\frac{\phi_{t}\phi_{r}}{\Delta^3}\left(\phi_{r}\phi_{tr}-\phi_{t}\phi_{tt}\right)\left(-3\left(1-\phi^2_{t}\right)\phi_{ttt}\phi_{tr}-3\phi_{t}\phi_{r}\phi_{ttt}\phi_{tt}\right)\\
&-\frac{\phi^2_{t}+\phi^2_{r}}{2\Delta^3}\left(\phi_{r}\phi_{tr}-\phi_{t}\phi_{tt}\right)\left(-3\left(1-\phi^2_{t}\right)\phi_{ttr}\phi_{tr}-3\phi_{t}\phi_{r}\phi_{ttr}\phi_{tt}\right),\\
\end{aligned}
\end{equation}
where
\begin{equation*}
\begin{aligned}
detC_m\left(\phi\left(t\right)\right)&:=\frac{1}{4\Delta^2}\left(\phi_{t}\phi_{ttt}-\phi_{r}\phi_{ttr}\right)^2+\frac{1}{4\Delta^2}\left(\phi_{t}\phi_{ttr}-\phi_{r}\phi_{ttt}\right)^2\\
&+\frac{1}{4\Delta^2}\left(\phi^2_{t}+\phi^2_{r}\right)\left(\phi_r\phi_{ttt}-\phi_{t}\phi_{ttr}\right)^2\\
&+\frac{1}{2\Delta^2}\left(\phi^2_{t}+\phi^2_{r}\right)\left(\phi_{ttt}\phi_{r}-\phi_{t}\phi_{ttr}\right)\left(\phi_{ttr}\phi_{r}-\phi_{ttt}\phi_{t}\right)\\
&+\frac{1}{2\Delta^2}\left(\phi^2_{t}+\phi^2_{r}\right)\left(\phi_{ttt}\phi_{r}-\phi_{t}\phi_{ttr}\right)\left(\phi_{ttt}\phi_{t}-\phi_{ttr}\phi_{r}\right)\\
&+\frac{1}{2\Delta^2}\phi^2_{t}\left(\phi_{ttr}\phi_{t}-\phi_{ttt}\phi_{r}\right)\left(\phi_{ttt}\phi_{r}-\phi_{t}\phi_{ttr}\right)\\
&+\frac{1}{2\Delta^2}\phi_{t}\phi_{r}\left(\phi_{ttr}\phi_{r}-\phi_{t}\phi_{ttt}\right)\left(\phi_{ttt}\phi_{r}-\phi_{t}\phi_{ttr}\right).\\
\end{aligned}
\end{equation*}

We consider the ``inner product" induced by the law of equilibrium \eqref{ph5} and \eqref{ph7}.

\begin{equation}
D\left(\phi\left(t\right)\right):=
\begin{pmatrix}
a^1+a^2+a^3 & b^1+b^2+b^3 \\
c^1+c^2            & d^1+d^2+d^3 
\end{pmatrix},
\end{equation}
where
\begin{equation*}
	a^1:=\frac{1}{2}\frac{r\phi^2_{ttt}\left(1+\phi^2_{r}\right)+r\phi^2_{ttr}\left(1-\phi^2_{t}\right)}{\Delta^{3/2}},
\end{equation*}

\begin{equation*}
	a^2:=r\phi_{ttr}\phi_{tr}\left(\frac{1-\phi^2_{t}}{\Delta^{3/2}}\right)_t,
\end{equation*}

\begin{equation*}
a^3:= r\phi_{ttr}\phi_{tt}\left(\frac{\phi_{r}\phi_{t}}{\Delta^{3/2}}\right)_t ,
\end{equation*}

\begin{equation*}
b^1:= \frac{r\phi_{ttt}\phi_{ttr}\left(1-\phi^2_{t}\right)+r\phi^2_{ttt}\phi_{t}\phi_{r}}{\Delta^{3/2}},
\end{equation*}

\begin{equation*}
b^2:=r\phi_{ttt}\phi_{tr}\left(\frac{1-\phi^2_{t}}{\Delta^{3/2}}\right)_t,
\end{equation*}

\begin{equation*}
b^3:=r\phi_{ttt}\phi_{tt}\left(\frac{\phi_{r}\phi_{t}}{\Delta^{3/2}}\right)_t,
\end{equation*}

\begin{equation*}
c^1:=\frac{r\phi_{tt}\phi_{tr}\left(1+\phi^2_{r}\right)-r\phi_{t}\phi_{r}\phi^2_{tr}}{\Delta^{3/2}},
\end{equation*}

\begin{equation*}
c^2:=\frac{1}{2}\frac{\phi_{t}\phi_{tt}\left(1+\phi^2_{r}\right)-\phi^2_{t}\phi_{r}\phi_{tr}}{\Delta^{3/2}},
\end{equation*}

\begin{equation*}
d^1:=\frac{1}{2}\frac{r\phi^2_{tr}\left(1-\phi^2_{t}\right)+r\phi^2_{tt}\left(1+\phi^2_{r}\right)}{\Delta^{3/2}},
\end{equation*}

\begin{equation*}
d^2:=\frac{1}{2}\frac{\phi_{t}\phi_{tr}\left(1-\phi^2_{t}\right)+\phi^2_{t}\phi_{r}\phi_{tt}}{\Delta^{3/2}},
\end{equation*}

\begin{equation*}
d^3:=\frac{1}{4}\frac{\phi^2_{t}}{r\Delta^{1/2}}.
\end{equation*}

By some simple calcaution, we have:
\begin{equation}\label{zch}
\begin{aligned}
&det D \left(\phi\left(t\right)\right)= a^1d^1-b^1c^1+a^1d^2+a^1d^3-c^1b^2-c^1b^3+a^2d^1+a^3d^1-b^1c^2\\
&+\left(a^2+a^3\right)\left(d^2+d^3\right)-\left(b^2+b^3\right)c^2\\
&= det D_m\left(\phi \left(t\right)\right)\\
&+\frac{r}{4\Delta^3}\phi_{t}\phi_{ttt}\left(1-\phi^2_{t}\right)\left(1+\phi^2_{r}\right)\left(\phi_{ttt}\phi_{tr}-\phi_{ttr}\phi_{tt}\right)\\
&+\frac{r}{4\Delta^3}\phi_{t}\phi_{ttr}\left(1-\phi^2_{t}\right)\left(\phi_{ttr}\phi_{tr}\left(1-\phi^2_{t}\right)-\phi_{ttt}\phi_{tt}\left(1+\phi^2_{r}\right)+2\phi_{ttr}\phi_{tr}\phi_{t}\phi_{r}\right)\\
&+\frac{r}{4\Delta^3}\phi^2_{t}\phi_{r}\phi_{ttr}\left(1-\phi^2_{t}\right)\left(\phi_{ttr}\phi_{tt}-\phi_{ttt}\phi_{tr}\right)\\
&+\frac{r}{4\Delta^3}\phi^2_{t}\phi_{r}\phi_{ttt}\left(\phi_{ttr}\phi_{tr}\left(1-\phi^2_{t}\right)-\phi_{ttt}\phi_{tt}\left(1+\phi^2_{r}\right)+2\phi_{ttt}\phi_{tr}\phi_{t}\phi_{r}\right)\\
&+\frac{r^2}{2\Delta^3}\phi_{tt}\phi_{tr}\left(\phi_{tt}\phi_{r}-\phi_{tr}\phi_{t}\right)\left(\phi_{ttr}\phi_{tr}\left(1-\phi^2_{t}\right)-\phi_{ttt}\phi_{tt}\left(1+\phi^2_{r}\right)+2\phi_{ttt}\phi_{tr}\phi_{t}\phi_{r}\right)\\
&+\frac{r^2}{2\Delta^3}\phi^2_{tt}\left(1+\phi^2_{r}\right)\left(\phi_{tt}\phi_{r}-\phi_{tr}\phi_{t}\right)\left(\phi_{ttr}\phi_{tt}-\phi_{ttt}\phi_{tr}\right)\\
&-\frac{3r^2}{2\Delta^4}\phi^2_{tr}\left(1-\phi^2_{t}\right)\left(\phi_{tr}\phi_{r}-\phi_{t}\phi_{tt}\right)\left(\phi_{ttr}\phi_{tr}\left(1-\phi^2_{t}\right)-\phi_{ttt}\phi_{tt}\left(1+\phi^2_{r}\right)+2\phi_{ttt}\phi_{tr}\phi_{t}\phi_{r}\right)\\
&-\frac{3r^2}{2\Delta^4}\phi_{tt}\phi_{tr}\left(1+\phi^2_{r}\right)\left(1-\phi^2_{t}\right)\left(\phi_{tr}\phi_{r}-\phi_{t}\phi_{tt}\right)\left(\phi_{ttr}\phi_{tt}-\phi_{tr}\phi_{ttt}\right)\\
&-\frac{3r^2}{2\Delta^4}\phi_{tt}\phi_{tr}\phi_{t}\phi_{r}\left(\phi_{tr}\phi_{r}-\phi_{t}\phi_{tt}\right)\left(\phi_{ttr}\phi_{tr}\left(1-\phi^2_{t}\right)-\phi_{ttt}\phi_{tt}\left(1+\phi^2_{r}\right)+2\phi_{ttt}\phi_{tr}\phi_{t}\phi_{r}\right)\\
&-\frac{3r^2}{2\Delta^4}\phi^2_{tt}\phi_{t}\phi_{r}\left(\phi_{r}\phi_{tr}-\phi_{t}\phi_{tt}\right)\left(1+\phi^2_{r}\right)\left(\phi_{ttr}\phi_{tt}-\phi_{tr}\phi_{ttt}\right)\\
&+\frac{r}{2\Delta^3}\left(\phi_{tt}\phi_{r}-\phi_{tr}\phi_{t}\right)\phi_{t}\phi_{tt}\left(\phi_{ttr}\phi_{tr}\left(1-\phi^2_{t}\right)-\phi_{ttt}\phi_{tt}\left(1+\phi^2_{r}\right)+2\phi_{ttt}\phi_{tr}\phi_{t}\phi_{r}\right)\\
&+\frac{r}{2\Delta^3}\left(\phi_{tt}\phi_{r}-\phi_{tr}\phi_{t}\right)\phi^2_{t}\phi_{r}\phi_{tt}\left(\phi_{tr}\phi_{ttt}-\phi_{tt}\phi_{ttr}\right)\\
&-\frac{3r}{2\Delta^4}\left(\phi_{tt}\phi_{r}-\phi_{tr}\phi_{t}\right)\phi_{tr}\phi_{t}\left(\phi_{ttr}\phi_{tr}\left(1-\phi^2_{t}\right)-\phi_{ttt}\phi_{tt}\left(1+\phi^2_{r}\right)+2\phi_{ttt}\phi_{tr}\phi_{t}\phi_{r}\right)\\
&-\frac{3r}{2\Delta^4}\left(\phi_{tt}\phi_{r}-\phi_{tr}\phi_{t}\right)\phi_{tr}\phi^2_{t}\phi_{r}\left(1-\phi^2_{t}\right)\left(\phi_{ttt}\phi_{tr}-\phi_{ttr}\phi_{tt}\right)\\
&-\frac{3r}{2\Delta^4}\left(\phi_{tt}\phi_{r}-\phi_{tr}\phi_{t}\right)\phi_{tt}\phi^2_{t}\phi_{r}\left(1-\phi^2_{t}\right)\left(\phi_{ttt}\phi_{tr}\left(1-\phi^2_{t}\right)-\phi_{ttt}\phi_{tt}\left(1+\phi^2_{r}\right)+2\phi_{ttt}\phi_{tr}\phi_{t}\phi_{r}\right)\\
&-\frac{3r}{4\Delta^4}\left(\phi_{tt}\phi_{r}-\phi_{tr}\phi_{t}\right)\phi_{tt}\phi^3_{t}\phi^2_{r}\left(\phi_{ttr}\phi_{tt}-\phi_{tr}\phi_{ttt}\right)\\
&+\frac{1}{8\Delta^2}\left(\phi^2_{t}\phi^2_{ttt}\left(1+\phi^2_{r}\right)+\phi^2_{t}\phi^2_{ttr}\left(1-\phi^2_{t}\right)\right)\\
&-\frac{1}{2\Delta^2}\phi^3_{t}\phi_{ttt}\phi_{tr}\\
&-\frac{3}{4\Delta^3}\phi_{ttr}\phi_{tr}\left(1-\phi^2_{t}\right)\phi^2_{t}\left(\phi_{tr}\phi_{r}-\phi_{t}\phi_{tt}\right)\\
&+\frac{1}{4\Delta^2}\phi^3_{t}\phi_{ttr}\phi_{tt}\phi_{tr}
+\frac{1}{4\Delta^2}\phi^2_{t}\phi_{r}\phi_{ttr}\phi^2_{tt}\\
&-\frac{3}{4\Delta^3}\phi^3_{t}\phi_{r}\phi_{ttr}\phi_{tt}\left(\phi_{r}\phi_{tr}-\phi_{t}\phi_{tt}\right),
\end{aligned}
\end{equation}
where
\begin{equation*}
	\begin{aligned}
	det D_m\left(\phi\left(t\right)\right) :=& \frac{r^2}{4\Delta^3}\left(1+\phi^2_{r}\right)\left(1-\phi^2_{t}\right)\left(\phi_{ttt}\phi_{tr}-\phi_{tt}\phi_{ttr}\right)^2\\
	&+\frac{r^2}{4\Delta^3}\left(\phi_{ttr}\phi_{tr}\left(1-\phi^2_{t}\right)-\phi_{ttt}\phi_{tt}\left(1+\phi^2_{r}\right)+2\phi_{ttt}\phi_{tr}\phi_{t}\phi_{r}\right).
	\end{aligned}
\end{equation*}

\section{ Energy induction}\label{ei}

In this section, we use the div-curl lemma \ref{dc} to  establish the  estimates for these  quantities $E_1\left(t\right)$, $E_2\left(t\right)$, $M_0\left(t\right)$, $M\left(t\right)$, $\xi_1\left(t\right)$, $\xi_2\left(t\right)$, $\eta_1\left(t\right)$ and $\eta_2\left(t\right)$ under some assumption as follow:
\begin{equation}\label{as}
	\|\phi_t\|_{L^{\infty}_{t,x}} + 	\|\phi_r\|_{L^{\infty}_{t,x}}\leq \sqrt{\eps}.
\end{equation}


We note that
\begin{equation}\label{tr}
	\begin{aligned}
	&\left(\phi_{tt}\phi_{rr}-\phi^2_{tr}\right)\\
	&=\frac{1}{1-\phi^2_{t}}\left(\left(1+\phi^2_{r}\right)\phi^2_{tt}-2\phi_{t
	}\phi_{r}\phi_{tt}\phi_{tr}-\left(1-\phi^2_{t}\right)\phi^2_{tr}\right)-\frac{\phi_{r}\phi_{tt}}{r}\frac{\Delta}{1-\phi^2_{t}}\\
	&=\frac{1}{1+\phi^2_{r}}\left(\left(1-\phi^2_{t}\right)\phi^2_{rr}+2\phi_{t}\phi_{r}\phi_{tr}\phi_{rr}-\left(1+\phi^2_{r}\right)\phi^2_{tr}\right)-\frac{\phi_{r}\phi_{rr}}{r}\frac{\Delta}{1+\phi^2_{r}}.
	\end{aligned}
\end{equation}

 Noting that 
\begin{equation}\label{bt1}
	E_1\left(t\right)\leq 2 \hat{E}_1\left(t\right),
\end{equation} 
 and using \eqref{ph1} , we have:
 \begin{equation}\label{bt2}
 \begin{aligned}
 \hat{E}_1\left(t\right)&=\hat{E}_1\left(0\right) + \int_{0}^{t}\int\lvert Q^1\rvert drdt\\
 &\leq\hat{E}_1\left(0\right)+2\left(\|\phi_t\|_{L^{\infty}_{t,x}} + 	\|\phi_r\|_{L^{\infty}_{t,x}}\right)M^{1/2}\left(t\right)\eta_1^{1/2}\left(t\right) +4M\left(t\right)\\
 &\leq\eps^2 +4\eps^{1/2}M^{1/2}\left(t\right)\eta_1^{1/2}\left(t\right) +4M\left(t\right).
 \end{aligned}
 \end{equation}

 Noting that 
\begin{equation}\label{bt3}
E_2\left(t\right)\leq 2 \hat{E}_1\left(t\right) + 2 \hat{E}_2\left(t\right),
\end{equation} 
and using \eqref{ph3}, we have:
\begin{equation}\label{bt4}
\begin{aligned}
	\hat{E}_2\left(t\right)&\leq \hat{E}_2\left(0\right) + \int_{0}^{t}\int\lvert W^1\rvert drdt \\
	&\leq\hat{E}_2\left(0\right)+2M_0^{1/2}\left(t\right)\xi_1^{1/2}\left(t\right) \\
		&\leq\eps^2 +2M_0^{1/2}\left(t\right)\xi_1^{1/2}\left(t\right)
\end{aligned}
\end{equation}

Noting that
\begin{equation}\label{bt5}
M\left(t\right)\leq 8M_1\left(t\right)+ 8M_2\left(t\right),
\end{equation} 
and using \eqref{ph5} , \eqref{ph6} and  \eqref{tr} , we have:
\begin{equation}
	\begin{aligned}
	M_1\left(t\right)\leq& \int_{0}^{t}\int P^1 drdt +8M^{1/2}_0\left(t\right)\xi^{1/2}_1\left(t\right)+8M^{1/2}\left(t\right)\xi^{1/2}_1\left(t\right)+16\eta^{1/2}_1\left(t\right)\xi^{1/2}_1\left(t\right)\\
	&+16\eta_1^{1/2}\left(t\right)M^{1/2}_0\left(t\right)+2\eta^{1/2}_1M^{1/2}\left(t\right)+\eps M_0^{1/2}\left(t\right)M^{1/2}\left(t\right)+\eps M\left(t\right)\\
	&\leq 2E_2\left(t\right)+2E^{1/2}_1\left(t\right)E^{1/2}_2\left(t\right)+2E_2\left(0\right)+2E^{1/2}_1\left(0\right)E^{1/2}_2\left(0\right)\\
	&+8M^{1/2}_0\left(t\right)\xi^{1/2}_1\left(t\right)+8M^{1/2}\left(t\right)\xi^{1/2}_1\left(t\right)\\
	&+16\eta^{1/2}_1\left(t\right)\xi^{1/2}_1\left(t\right)+16\eta_1^{1/2}\left(t\right)M^{1/2}_0\left(t\right)+2\eta^{1/2}_1M^{1/2}\left(t\right)\\&+\eps M_0^{1/2}\left(t\right)M^{1/2}\left(t\right)+\eps M\left(t\right)\\
	&\leq 4\eps^2+2E_2\left(t\right)+2E^{1/2}_1\left(t\right)E^{1/2}_2\left(t\right)\\ &+8M^{1/2}_0\left(t\right)\xi^{1/2}_1\left(t\right)+8M^{1/2}\left(t\right)\xi^{1/2}_1\left(t\right)+16\eta^{1/2}_1\left(t\right)\xi^{1/2}_1\left(t\right)\\
	&+16\eta_1^{1/2}\left(t\right)M^{1/2}_0\left(t\right)+2\eta^{1/2}_1M^{1/2}\left(t\right)+\eps M_0^{1/2}\left(t\right)M^{1/2}\left(t\right)+\eps M\left(t\right),\\
	\end{aligned}
\end{equation}

\begin{equation}
	\begin{aligned}
	M_2\left(t\right)&\leq\int_{0}^{t}\int P^2 drdt + 2\eps M^{1/2}\left(t\right)+2\eps M^{1/2}\left(t\right)+\eps^{1/2}M\left(t\right)\\
	&+2M^{1/2}_0\left(t\right)\xi_1^{1/2}\left(t\right)+2M^{1/2}_0\left(t\right)\xi_1^{1/2}\left(t\right)+M^{1/2}\left(t\right)\xi^{1/2}_1\left(t\right)\\
	&+2M_0\left(t\right)+\eps M_0^{1/2}\left(t\right)M^{1/2}\left(t\right)+2M_0\left(t\right)+2\eps M\left(t\right)+\eps M_0^{1/2}\left(t\right)M^{1/2}\left(t\right)\\
	&\leq2E_2\left(t\right)+2E^{1/2}_1\left(t\right)E^{1/2}_2\left(t\right)+2E_2\left(0\right)+2E^{1/2}_1\left(0\right)E^{1/2}_2\left(0\right)\\
	&+4\eps M^{1/2}\left(t\right)+\eps^{1/2}M\left(t\right)+4M^{1/2}_0\left(t\right)\xi_1^{1/2}\left(t\right)+M^{1/2}\left(t\right)\xi^{1/2}_1\left(t\right)\\
	&+4M_0\left(t\right)+2\eps M_0^{1/2}\left(t\right)M^{1/2}\left(t\right)+2\eps M\left(t\right)\\
    &\leq 4\eps^2 +2E_2\left(t\right)+2E^{1/2}_1\left(t\right)E^{1/2}_2\left(t\right)\\
	&+4\eps M^{1/2}\left(t\right)+\eps^{1/2}M\left(t\right)+4M^{1/2}_0\left(t\right)\xi_1^{1/2}\left(t\right)+M^{1/2}\left(t\right)\xi^{1/2}_1\left(t\right)\\
	&+4M_0\left(t\right)+2\eps M_0^{1/2}\left(t\right)M^{1/2}\left(t\right)+2\eps M\left(t\right).\\
	\end{aligned}
   \end{equation}

Besides, we can obtain:
\begin{equation}
	\begin{aligned}
	\int_{0}^{t}\int \lvert P^1 \rvert drdt\leq& M_1\left(t\right) +8M^{1/2}_0\left(t\right)\xi^{1/2}_1\left(t\right)+8M^{1/2}\left(t\right)\xi^{1/2}_1\left(t\right)+16\eta^{1/2}_1\left(t\right)\xi^{1/2}_1\left(t\right)\\
	&+16\eta_1^{1/2}\left(t\right)M^{1/2}_0\left(t\right)+2\eta^{1/2}_1M^{1/2}\left(t\right)+\eps M_0^{1/2}\left(t\right)M^{1/2}\left(t\right)+\eps M\left(t\right)
	\end{aligned}
\end{equation}

\begin{equation}
	\begin{aligned}
	\int_{0}^{t}\int \lvert P^2 \rvert drdt\leq& M_2\left(t\right)+4\eps M^{1/2}\left(t\right)+\eps^{1/2}M\left(t\right)+4M^{1/2}_0\left(t\right)\xi_1^{1/2}\left(t\right)+M^{1/2}\left(t\right)\xi^{1/2}_1\left(t\right)\\
	&+4M_0\left(t\right)+2\eps M_0^{1/2}\left(t\right)M^{1/2}\left(t\right)+2\eps M\left(t\right) .\\
	\end{aligned}
\end{equation}

From \eqref{dc3}, we know:
\begin{equation}
	\begin{aligned}
	\xi_1\left(t\right) + \frac{1}{16}M_{0,1}\left(t\right) \leq\xi_2\left(t\right)+\eps M^{1/2}_{0,1}\left(t\right)\xi^{1/2}_1\left(t\right)+\eps M^{1/2}_{0,1}\left(t\right)\xi^{1/2}_1\left(t\right)+\eps M_{0,1}\left(t\right).
	\end{aligned}
\end{equation}

Noting that
\begin{equation}
	\begin{aligned}
	&\phi^2_{r}\phi^2_{tt}\lesssim\left(\phi_{r}\phi_{tt}-\phi_{t}\phi_{tr}\right)^2 + \phi^2_{t}\phi^2_{tr}\\
	 &\phi^2_{r}\phi^2_{tr}\lesssim\left(\phi_{r}\phi_{tr}-\phi_{t}\phi_{tt}\right)^2 + \phi^2_{t}\phi^2_{tt},
	\end{aligned}
\end{equation}

So we have:
\begin{equation}
\begin{aligned}
M_{0,2}\left(t\right) &\lesssim M_{0,1}\left(t\right) + \eta_2\left(t\right) +2\eps^{1/2}M^{1/2}\left(t\right)M_{0,1}^{1/2}\left(t\right)\\
&+2\eps^{3/2}M^{1/2}\left(t\right)M_{0,1}^{1/2}\left(t\right)+\eps M_{0,1}\left(t\right).
\end{aligned}
\end{equation}

By definition, we know:
\begin{equation}
	M_0\left(t\right) = M_{0,1}\left(t\right)+M_{0,2}\left(t\right) .
\end{equation}

Using the div-curl lemma \ref{dc}, we have:
 \begin{equation}
 	\begin{aligned}
& \xi_2\left(t\right)\\
 &\lesssim\left(\hat{E}_2\left(0\right)  + \sup\limits_{0\leq t \leq T} \hat{E}_2\left(t\right) + M_0^{1/2}\left(t\right)\xi_1^{1/2}\left(t\right)\right)\\
 &\cdot\left(\hat{E}_2\left(0\right)+\hat{E}^{1/2}_1\left(0\right)\hat{E}^{1/2}_2\left(0\right) +\sup\limits_{0\leq t \leq T}\left(\hat{E}_2\left(t\right)+\hat{E}^{1/2}_1\left(t\right)\hat{E}^{1/2}_2\left(t\right)\right)+\int_{0}^{t}\int \lvert P^1 \rvert drdt\right)\\
 &\lesssim\left(\eps^2  + \sup\limits_{0\leq t \leq T} \hat{E}_2\left(t\right) + M_0^{1/2}\left(t\right)\xi_1^{1/2}\left(t\right)\right)\\
 &\cdot\left(\eps^2 +\sup\limits_{0\leq t \leq T}\left(\hat{E}_2\left(t\right)+\hat{E}^{1/2}_1\left(t\right)\hat{E}^{1/2}_2\left(t\right)\right)+\int_{0}^{t}\int \lvert P^1 \rvert drdt\right),
 	\end{aligned}
 \end{equation}

\begin{equation}
\begin{aligned}
	&\eta_2\left(t\right) \\
	&\lesssim\left(\hat{E}_1\left(0\right) +\sup\limits_{0\leq t \leq T} \hat{E}_1\left(t\right)+ \eps^{1/2}M^{1/2}\left(t\right)\eta_1^{1/2}\left(t\right) +M\left(t\right) \right)\\
	&\cdot\left(\hat{E}_2\left(0\right)+\hat{E}^{1/2}_1\left(0\right)\hat{E}^{1/2}_2\left(0\right) +\sup\limits_{0\leq t \leq T}\left(\hat{E}_2\left(t\right)+\hat{E}^{1/2}_1\left(t\right)\hat{E}^{1/2}_2\left(t\right)\right)+\int_{0}^{t}\int \lvert P^1 \rvert drdt\right)\\
	&\lesssim\left(\eps^2 +\sup\limits_{0\leq t \leq T} \hat{E}_1\left(t\right)+ \eps^{1/2}M^{1/2}\left(t\right)\eta_1^{1/2}\left(t\right) +M\left(t\right) \right)\\
	&\cdot\left(\eps^2 +\sup\limits_{0\leq t \leq T}\left(\hat{E}_2\left(t\right)+\hat{E}^{1/2}_1\left(t\right)\hat{E}^{1/2}_2\left(t\right)\right)+\int_{0}^{t}\int \lvert P^1 \rvert drdt\right).
\end{aligned}
\end{equation}

Based above and standard bootstrap arguments, we obtain:
\begin{equation}\label{sme}
\begin{aligned}
	&E_1\left(t\right)\lesssim \eps^2,\\
	&E_2\left(t\right)\lesssim \eps^2,\\
	&M_0\left(t\right)\lesssim \eps^2,\\
	&M\left(t\right)\lesssim\eps^2,\\
	&\xi_1\left(t\right)\lesssim \eps^4,\\
	&\xi_2\left(t\right)\lesssim \eps^4,\\
	&\eta_1\left(t\right)\lesssim \eps^4,\\
	&\eta_2\left(t\right)\lesssim \eps^4.
\end{aligned}
\end{equation}

Under  the assumption $\phi_t|_{r=0}=0$, and 
\begin{equation*}
\int_{0}^{\infty} \frac{\phi^2_t}{r}  \big|_{t=0} dr \lesssim \eps^2, 
\end{equation*}
we have:
\begin{equation}
\begin{aligned}
\phi^2_{t}&\leq  \int_{0}^{\infty} \phi^2_{tr} r dr +  \int_{0}^{\infty} \frac{\phi^2_t}{r} dr\\
&\leq E_1\left(t\right)+E_2\left(t\right)\\
&\lesssim \eps^2.
\end{aligned}
\end{equation}

 Similarly,
 \begin{equation}
\begin{aligned}
\phi^2_r&\leq \int_{0}^{\infty}\phi^2_{rr} r dr  + \int_{0}^{\infty}\frac{\phi^2_r}{r}dr\\
&\leq E_1\left(t\right)+E_2\left(t\right) \\
&\lesssim \eps^2.
\end{aligned}
\end{equation}

Based above, we obtain:
\begin{equation}
\|\phi_t\|_{L^{\infty}_{t,x}} + 	\|\phi_r\|_{L^{\infty}_{t,x}}\lesssim \eps.
\end{equation}

According to the  continuous induction arguments, we know  that the following inequality is actually true.
\begin{equation}
\|\phi_t\|_{L^{\infty}_{t,x}} + 	\|\phi_r\|_{L^{\infty}_{t,x}}\lesssim \sqrt{\eps}
\end{equation}

\section{Global existence of the solution}\label{77}
In this section, we will prove  that the energy  of the third derivative is uniformly bounded. Combining this and local existence results of Smith-Tataru \cite{Smith 2005}, we obatin the global existence of the solution.

Using the Sobolev embedding, we know:
\begin{equation}\label{s1}
\begin{aligned}
&\left(\int r \phi^4_{tt} dr\right)^{1/4} \leq \left(\int r \phi^2_{tt} dr\right)^{1/4}\left(\int r \phi^2_{ttr} dr\right)^{1/4} \leq \eps^{1/2} E^{1/4}_{3,q}\left(t\right),\\
&\left(\int r \phi^4_{tr} dr\right)^{1/4} \leq \left(\int r \phi^2_{tr} dr\right)^{1/4}\left(\int r \phi^2_{trr} dr\right)^{1/4} \leq \eps^{1/2} E^{1/4}_{3,s}\left(t\right),\\
&\left(\int r \phi^4_{rr} dr\right)^{1/4} \leq \left(\int r \phi^2_{rr} dr\right)^{1/4}\left(\int r \phi^2_{rrr} dr\right)^{1/4} \leq \eps^{1/2} E^{1/4}_{3,l}\left(t\right).
\end{aligned}
\end{equation}

Similarly,
\begin{equation}\label{s2}
	\begin{aligned}
	&\left(\int r\phi^6_{tt} dr \right)^{1/6}\leq\left(\int r\phi^2_{tt} dr \right)^{1/6}\left(\int r \phi^2_{ttr} dr \right)^{1/3}\leq 2\eps^{1/3}E^{1/3}_{3,q}\left(t\right),\\
	&\left(\int r\phi^6_{tr} dr \right)^{1/6}\leq\left(\int r\phi^2_{tr} dr \right)^{1/6}\left(\int r \phi^2_{trr} dr \right)^{1/3}\leq 2\eps^{1/3}E^{1/3}_{3,s}\left(t\right),\\
	&\left(\int r\phi^6_{rr} dr \right)^{1/6}\leq\left(\int r\phi^2_{rr} dr \right)^{1/6}\left(\int r \phi^2_{rrr} dr \right)^{1/3}\leq 2\eps^{1/3}E^{1/3}_{3,l}\left(t\right).
	\end{aligned}
\end{equation}

Using Newton-Leibniz formula, we have:
\begin{equation}\label{wqes}
\begin{aligned}
&r\phi^2_{tt} \leq 2\int r \phi_{tt}\phi_{ttr} dr \leq 2 E^{1/2}_2\left(t\right) E^{1/2}_{3,q}\left(t\right)\leq 2\eps E^{1/2}_{3,q}\left(t\right),\\
&r\phi^2_{tr} \leq 2\int r \phi_{tr}\phi_{trr} dr \leq 2 E^{1/2}_2\left(t\right) E^{1/2}_{3,s}\left(t\right) \leq 2\eps E^{1/2}_{3,s}\left(t\right),\\
&r\phi^2_{rr} \leq 2\int r \phi_{rr}\phi_{rrr} dr \leq 2 E^{1/2}_2\left(t\right) E^{1/2}_3\left(t\right)\leq  2\eps E^{1/2}_{3,l}\left(t\right).
\end{aligned}
\end{equation}

We note:
\begin{equation*}
\begin{aligned}
&\left(1-\phi^2_{t}\right)\phi_{trr}+\frac{\phi_{tr} \Delta}{r}-2\phi_{t}\phi_{tt}\phi_{rr}+\frac{\phi_{r}\left(\phi_{r}\phi_{tr}-\phi_{t}\phi_{tt}\right)}{r} \\
&=\left(1+\phi^2_{r}\right)\phi_{ttt}-2\phi_{t}\phi_{tr}\phi_{tr}-2\phi_{t}\phi_{r}\phi_{ttr},\\
\end{aligned}
\end{equation*}

\begin{equation}
	\begin{aligned}
	&r\left(1-\phi^2_{t}\right)^2\phi^2_{trr}+\frac{\phi^2_{tr}\Delta^2}{r}=-2\left(1-\phi^2_{t}\right)\phi_{trr}\phi_{tr}\Delta+ rS^2_{trr},
	\end{aligned}
\end{equation}
where 
\begin{equation*}
\begin{aligned}
S_{trr}:=&\left(1+\phi^2_{r}\right)\phi_{ttt}-2\phi_{t}\phi_{tr}\phi_{tr}-2\phi_{t}\phi_{r}\phi_{ttr}\\
&+2\phi_{t}\phi_{tt}\phi_{rr}-\frac{\phi_{r}\left(\phi_{r}\phi_{tr}-\phi_{t}\phi_{tt}\right)}{r}.
\end{aligned}
\end{equation*}

By integration by parts  and  combining \eqref{s1} and \eqref{s2}, we know:
\begin{equation}
	\begin{aligned}
	&-2\int\left(1-\phi_{t}\right)\Delta\phi_{trr}\phi_{tr} dr \\
	&= \int -2 \phi_{t}\phi^3_{tr}\Delta dr + \int 2\left(1-\phi^2_{t}\right)\left(\phi_{r}\phi_{rr}\phi^2_{tr}\right) dr-2\int \left(1-\phi^2_{t}\right)\left(\phi_{t}\phi^3_{tr}\right) dr\\ 
	&\leq 4 \int \lvert \phi^3_{tr} \phi_{t}\rvert dr + 4 \int \lvert \phi^2_{tr}\phi_{rr} \phi_{t}\rvert dr\\
	 &\leq 4\left(\int \frac{\phi^2_{t}}{r} dr \right)^{1/2}\left(\int r \phi^6_{tr} dr \right)^{1/2}+ 4\left(\int \frac{\phi^2_{r}}{r} dr \right)^{1/2}\left(\int r \phi^6_{tr} dr \right)^{1/3}\left(\int r \phi^6_{rr} dr \right)^{1/6}\\
	 & \leq 16\eps^2 E_{3,s}\left(t\right)+16\eps^2 E^{2/3}_{3,s}\left(t\right) E^{1/3}_{3,l}\left(t\right).
	\end{aligned} 
\end{equation}

Using hardy inequality and \eqref{wqes}, we obtain:
\begin{equation}
\begin{aligned}
&\int\frac{\phi^2_{r}\phi^2_{t}\phi^2_{tt}}{r} dr \\
&\leq \|r\phi^2_{tt}\|_{L^{\infty}}\left(\int \frac{\phi^4_{t}}{r^2}  dr + \int \frac{\phi^4_{r}}{r^2} dr \right)\\
&\leq 4 \|r\phi^2_{tt}\|_{L^{\infty}} \int \phi^2_{t}\phi^2_{tr}+ \phi^2_{r}\phi^2_{rr} dr\\
&\leq 4 \|r\phi^2_{tt}\|_{L^{\infty}} \left(\left(\int \frac{\phi^4_{t}}{r} dr \right)^{1/2}\left(\int r \phi^4_{tr} dr\right)^{1/2}  + \left(\int \frac{\phi^4_{r}}{r} dr \right)^{1/2}\left(\int r \phi^4_{rr} dr\right)^{1/2} \right)\\
&\leq 8\eps^3 E^{1/2}_{3,q}\left(t\right)\left( E^{1/2}_{3,s}\left(t\right) + E^{1/2}_{3,l}\left(t\right) \right).
\end{aligned}
\end{equation}

Combining above, we can get:
\begin{equation}
\begin{aligned}
E_{3,s}\left(t\right) &\leq E_{3,q}\left(t\right)+ 4\eps^3 E_{3,s}\left(t\right)+4\eps^3 E^{1/2}_{3,q}\left(t\right) E^{1/2}_{3,l}\left(t\right)+4\eps^2E_{3,q}\\
&+\eps^2 E_{3,s}\left(t\right)+16\eps^2 E_{3,s}\left(t\right)+16\eps^2 E^{2/3}_{3,s}\left(t\right) E^{1/3}_{3,l}\left(t\right)\\
&+8\eps^3 E^{1/2}_{3,q}\left(t\right)\left( E^{1/2}_{3,s}\left(t\right) + E^{1/2}_{3,l}\left(t\right) \right).
\end{aligned}
\end{equation}

We note:
\begin{equation*}
	\begin{aligned}
	&\left(1-\phi^2_{t}\right)\phi_{rrr}+ \frac{\phi_{rr}\Delta}{r} - \frac{\phi_{r}\Delta}{r^2}=\left(1+\phi^2_{r}\right)\phi_{ttr}+2\phi_{r}\phi_{rr}\phi_{tt}\\
	&-2\phi_{r}\phi_{tr}\phi_{tr}-2\phi_{t}\phi_{r}\phi_{trr}-\frac{\phi_{r}\left(\phi_{r}\phi_{rr}-\phi_{t}\phi_{tr}\right)}{r},
	\end{aligned}
\end{equation*}

\begin{equation}
	\begin{aligned}
    r\left(1-\phi^2_{t}\right)^2\phi^2_{rrr}+r\left(\frac{\phi_{rr}}{r}-\frac{\phi_{r}}{r^2}\right)^2\Delta^2+2r\left(1-\phi^2_{t}\right)\Delta\phi_{rrr}\left(\frac{\phi_{rr}}{r}-\frac{\phi_{r}}{r^2}\right)=rS^2_{rrr},
	\end{aligned}
\end{equation}
where
\begin{equation*}
	\begin{aligned}
    S_{rrr}:=&\left(1+\phi^2_{r}\right)\phi_{ttr}+2\phi_{r}\phi_{rr}\phi_{tt}\\
    &-2\phi_{r}\phi_{tr}\phi_{tr}-2\phi_{t}\phi_{r}\phi_{trr}-\frac{\phi_{r}\left(\phi_{r}\phi_{rr}-\phi_{t}\phi_{tr}\right)}{r}.
	\end{aligned}
\end{equation*}

By integration by parts  and  combining with \eqref{s1} and \eqref{s2}, we have:
\begin{equation}
	\begin{aligned}
	&2 \int r\left(1-\phi^2_{t}\right)\Delta \phi_{rrr}  \frac{\phi_{rr}}{r} dr\\
	& =\int 2\phi_{t}\phi_{tt}\phi^2_{rr} dr -\int 2 \left(1-\phi^2_{t}\right) \phi_{r}\phi^3_{rr} dr  + \int 2\left(1-\phi^2_{t}\right)\phi_{t}\phi_{tr}\phi^2_{rr} dr \\
	&\leq 2\left(\int \frac{\phi^2_{t}}{r} dr \right)^{1/2}\left(\int r \phi^6_{rr} dr \right)^{1/3}\left(\int r \phi^6_{tt} dr \right)^{1/6}
	+ 2\left(\int \frac{\phi^2_{t}}{r} dr \right)^{1/2}\left(\int r \phi^6_{rr} dr \right)^{1/2}\\
	&+ 2 \left(\int \frac{\phi^2_{t}}{r} dr \right)^{1/2}\left(\int r \phi^6_{rr} dr \right)^{1/3}\left(\int r \phi^6_{tr} dr \right)^{1/6}
	\\
	&\leq 2\eps^2 E^{2/3}_{3,l}\left(t\right)E^{1/3}_{3,q}\left(t\right)+ 2\eps^2 E_{3,l}\left(t\right)+ 2\eps^2 E^{2/3}_{3,l}\left(t\right)E^{1/3}_{3,s}\left(t\right).
	\end{aligned}
\end{equation}

Using integration by parts, we can obtain:
\begin{equation}
	\begin{aligned}
&	-2\int\left(1-\phi^2_{t}\right)\Delta\phi_{rrr}\frac{\phi^2_{r}}{r} dr\\
& =2\int \left(1-\phi^2_{t}\right)\Delta \frac{\phi^2_{rr}}{r} dr -2\int \left(1-\phi^2_{t}\right)\Delta\frac{\phi_{rr}\phi_{r}}{r^2}dr
	+ 2\int \left(\left(1-\phi^2_{t}\right)\Delta\right)_r\frac{\phi_{rr}\phi_{r}}{r}dr\\
	&=2\int \left(1-\phi^2_{t}\right)\Delta \frac{\phi^2_{rr}}{r} dr -4\int \left(1-\phi^2_{t}\right)\Delta\frac{\phi_{rr}\phi_{r}}{r^2} dr + 2\int \left(1-\phi^2_{t}\right)\Delta\frac{\phi_{rr}\phi_{r}}{r^2} dr \\
	&+2\int \left(\left(1-\phi^2_{t}\right)\Delta\right)_r\frac{\phi_{rr}\phi_{r}}{r} dr \\
	&=2\int \left(1-\phi^2_{t}\right)\Delta \frac{\phi^2_{rr}}{r} dr -4\int \left(1-\phi^2_{t}\right)\Delta\frac{\phi_{rr}\phi_{r}}{r^2} dr+2 \int \left(1-\phi^2_{t}\right)\Delta \frac{\phi^2_{r}}{r^3} dr\\
	& +2\int \left(\left(1-\phi^2_{t}\right)\Delta\right)_r\frac{\phi_{rr}\phi_{r}}{r} dr - \int \left(\left(1-\phi^2_{t}\right)\Delta\right)_r\frac{\phi^2_{r}}{r^2} dr \\
	&= 2 \int r \left(1-\phi_{t}\right)\Delta \left(\frac{\phi_{r}\phi_{rr}}{r}-\frac{\phi_{r}}{r^2}\right)^2 dr+\int \left(\left(1-\phi^2_{t}\right)\Delta\right)_r\frac{\phi_{rr}\phi_{r}}{r} dr\\
	&+ \int \left(\left(1-\phi^2_{t}\right)\Delta\right)_r \phi_{r}\left(\frac{\phi_{rr}}{r}-\frac{\phi^2_{r}}{r^2}\right) dr.
	\end{aligned}
\end{equation}

By \eqref{wqes},  h\"older and hardy inequality, we can get:
\begin{equation}
\begin{aligned}
&\int \left(\left(1-\phi^2_{t}\right)\Delta\right)_r \phi_{r}\left(\frac{\phi_{rr}}{r}-\frac{\phi^2_{r}}{r^2}\right) dr\\
&\leq 4 \left(\int \frac{\phi_{t}\phi_{r}\phi_{tr}+\phi^2_{r}\phi_{rr}}{r} dr \right)^{1/2}\left(\int r \left(\frac{\phi_{rr}}{r}-\frac{\phi^2_{r}}{r^2}\right)^2 dr \right)^{1/2}\\
&\leq 4 \left(\|r\phi^2_{tr}\|_{L^{\infty}}+\|r\phi^2_{rr}\|_{L^{\infty}}\right)^{1/2}\left(\int \frac{\phi^4_{t}}{r^2} dr + \int \frac{\phi^4_{t}}{r^2}  dr \right)^{1/2} E^{1/2}_{3,l}\left(t\right)\\
&\leq 8\eps^{3/2} \left(E^{1/4}_{3,s}\left(t\right)+E^{1/4}_{3,l}\left(t\right)\right)\left( E^{1/4}_{3,s}\left(t\right) + E^{1/4}_{3,l}\left(t\right) \right) E^{1/2}_{3,l}\left(t\right).
\end{aligned}
\end{equation}

By some simple calcaution, we know:
\begin{equation}
\begin{aligned}
\int \left(\left(1-\phi^2_{t}\right)\Delta\right)_r\frac{\phi_{rr}\phi_{r}}{r} dr=& 2 \int \left(1-\phi^2_{t}\right) \frac{\phi^2_{rr}\phi^2_{r}}{r} dr -2 \int \Delta \frac{\phi_{t}\phi_{tr}\phi_{r}\phi_{rr}}{r} dr \\
&-2\int \left(1-\phi^2_{t} \right)\frac{\phi_{t}\phi_{tr}\phi_{r}\phi_{rr}}{r} dr.
\end{aligned}
\end{equation}

From \eqref{s1}, and using h\"older inequality, we know:
\begin{equation*}
\begin{aligned}
	&\int \frac{\phi_{t}\phi_{tr}\phi_{r}\phi_{rr}}{r} dr\\
	&=\int r \phi_{t}\phi_{tr}\phi_{rr}\left(\frac{\phi_{rr}}{r}-\frac{\phi^2_{r}}{r^2}\right) dr +\int \phi_{t}\phi_{tr}\phi^2_{rr} dr \\
\end{aligned}
\end{equation*}
\begin{equation}
\begin{aligned}
	&\leq\left(\int r \left(\frac{\phi_{rr}}{r}-\frac{\phi^2_{r}}{r^2} \right)^2 dr \right)^{1/2}\left(\int r\phi^2_{t}\phi^2_{tr}\phi^2_{rr} dr \right)^{1/2}\\
	&+\left(\int \frac{\phi^2_{t}}{r} dr \right)^{1/2}\left(\int r \phi^6_{rr} dr \right)^{1/3}\left(\int r \phi^6_{tr} dr \right)^{1/6}\\
	&\leq 2\eps^{3/2}E^{1/2}_{3,l}\left(t\right)E^{1/4}_{3,s}\left(t\right)E^{1/4}_{3,l}\left(t\right)+2\eps^2E^{2/3}_{3,l}\left(t\right)E^{1/3}_{3,s}\left(t\right),
\end{aligned}
\end{equation}

Noting \eqref{s1} and \eqref{wqes}, we use h\"older and hardy inequality to obtain:
\begin{equation}
\begin{aligned}
&\int\frac{\phi^2_{r}\phi^2_{t}\phi^2_{tr}}{r} dr \\
&\leq \|r\phi^2_{tr}\|_{L^{\infty}}\left(\int \frac{\phi^4_{t}}{r^2} dr  + \int \frac{\phi^4_{r}}{r^2} dr \right)\\
&\leq 4 \|r\phi^2_{tr}\|_{L^{\infty}} \int \phi^2_{t}\phi^2_{tr}+ \phi^2_{r}\phi^2_{rr} dr\\
&\leq 4 \|r\phi^2_{tr}\|_{L^{\infty}} \left(\left(\int \frac{\phi^4_{t}}{r} dr \right)^{1/2}\left(\int r \phi^4_{tr} dr\right)^{1/2}  + \left(\int \frac{\phi^4_{r}}{r} dr \right)^{1/2}\left(\int r \phi^4_{rr} dr\right)^{1/2} \right)\\
&\leq 8\eps^3 E^{1/2}_{3,s}\left(t\right)\left( E^{1/2}_{3,s}\left(t\right) + E^{1/2}_{3,l}\left(t\right) \right).
\end{aligned}
\end{equation}

Similarly,
\begin{equation}
\begin{aligned}
\int\frac{\phi^4_{r}\phi^2_{rr}}{r} dr &\leq \|r\phi^2_{rr}\|_{L^{\infty}}\int \frac{\phi^4_{r}}{r^2} dr \\
&\leq 4 \|r\phi^2_{rr}\|_{L^{\infty}} \int  \phi^2_{r}\phi^2_{rr} dr\\
&\leq 4 \|r\phi^2_{rr}\|_{L^{\infty}}  \left(\int \frac{\phi^4_{r}}{r} dr \right)^{1/2}\left(\int r \phi^4_{rr} dr\right)^{1/2} \\
&\leq 8\eps^3 E^{1/2}_{3,l}\left(t\right) E^{1/2}_{3,l}\left(t\right). 
\end{aligned}
\end{equation}

Based above,  we can get:
\begin{equation}
\begin{aligned}
	E_{3,l}\left(t\right) &\leq E_{3,q}\left(t\right)+4\eps^3 E^{1/2}_{3,q}\left(t\right) E^{1/2}_{3,l}\left(t\right)+4\eps^3 E_{3,s}\left(t\right)\\
	&+4\eps^2 E_{3,s}\left(t\right)+ 8\eps^3 E^{1/2}_{3,s}\left(t\right)\left( E^{1/2}_{3,s}\left(t\right) + E^{1/2}_{3,l}\left(t\right) \right)+8\eps^3 E^{1/2}_{3,l}\left(t\right) E^{1/2}_{3,l}\left(t\right)\\
	&+2\eps^{3/2}E^{1/2}_{3,l}\left(t\right)E^{1/4}_{3,s}\left(t\right)E^{1/4}_{3,l}\left(t\right)+2\eps^2E^{2/3}_{3,l}\left(t\right)E^{1/3}_{3,s}\left(t\right)\\
	&+8\eps^{3/2} \left(E^{1/4}_{3,s}\left(t\right)+E^{1/4}_{3,l}\left(t\right)\right)\left( E^{1/4}_{3,s}\left(t\right) + E^{1/4}_{3,l}\left(t\right) \right) E^{1/2}_{3,l}\left(t\right). 
\end{aligned}	
\end{equation}

It is obvious that
\begin{equation}
E_3\left(t\right) = E_{3,q}\left(t\right) + E_{3,s}\left(t\right) + E_{3,l}\left(t\right),
\end{equation}

\begin{equation}
E_{3,q}\left(t\right)\leq 4\tilde{E}_3\left(t\right), 
\end{equation} 
and 
\begin{equation}
	\begin{aligned}
	\hat{E}_3\left(t\right) &\leq \tilde{E}_3\left(t\right) + 3\eps^{3/2} E^{3/4}_{3,q}\left(t\right) E^{1/4}_{3,s}\left(t\right) + \eps^{3/2} E^{1/2}_{3,q}\left(t\right) E^{1/2}_{3,s}\left(t\right) + \eps^{3/2} E_{3,q}\left(t\right)\\
	&+\eps^{5/2} E^{3/4}_{3,q}\left(t\right) E^{1/4}_{3,s}\left(t\right)+\eps^{5/2} E_{3,q}\left(t\right).
	\end{aligned}
\end{equation}

By integration by parts, we obtain:
\begin{equation*}
	\begin{aligned}
	\int_{0}^{t} \int \phi^4_{tt} drdt &=\int \phi_{t}\phi^3_{tt} dr\bigg|^{t}_0 -3 \int_{0}^{t} \int\phi_{t}\phi^2_{tt}\phi_{ttt} drdt\\
	&\leq \frac{1}{2} \int_{0}^{t}\int \phi^4_{tt} drdt + 12\int_{0}^{t}\int \phi^2_{t}\phi^2_{ttt} drdt +\int \phi_{t}\phi^3_{tt} dr\bigg|^{t}_0,
	\end{aligned}
\end{equation*}

\begin{equation*}
	\int \phi_{t}\phi^3_{tt} dr \leq \left(\int \frac{\phi^2_{t}}{r} dr \right)^{1/2}\left(\int r \phi^6_{tt} dr \right)^{1/2}\leq \eps^2 E_{3,q}\left(t\right),
\end{equation*}

\begin{equation}\label{s41}
\int_{0}^{t} \int \phi^4_{tt} drdt \leq 2\eps^2 E_{3,q}\left(t\right) + 24	M_h\left(t\right).
\end{equation}

Similarly, 
\begin{equation*}
\begin{aligned}
\int_{0}^{t} \int \phi^4_{tr} drdt &=\int \phi_{r}\phi^3_{tr} dr\bigg|^{t}_0 -3 \int_{0}^{t} \int\phi_{r}\phi^2_{tr}\phi_{ttr} drdt\\
&\leq \frac{1}{2} \int_{0}^{t}\int \phi^4_{tr} drdt + 12\int_{0}^{t}\int \phi^2_{t}\phi^2_{ttr} drdt +\int \phi_{r}\phi^3_{tr} dr\bigg|^{t}_0,
\end{aligned}
\end{equation*}

\begin{equation*}
\int \phi_{r}\phi^3_{tr} dr \leq \left(\int \frac{\phi^2_{r}}{r} dr \right)^{1/2}\left(\int r \phi^6_{tr} dr \right)^{1/2}\leq \eps^2, E_{3,s}\left(t\right)
\end{equation*}

\begin{equation}\label{s42}
\int_{0}^{t} \int \phi^4_{tr} drdt \leq 2\eps^2 E_{3,s}\left(t\right) + 24	M_h\left(t\right).
\end{equation}

Besides, we have:
\begin{equation}\label{cha1}
	\begin{aligned}
 &\left(\int_{0}^{t}\int\left(\phi_{r}\phi_{tr}-\phi_{t}\phi_{tt}\right)^4 dr dt\right)^{1/4}\\
 &\leq 16 \left(\int \phi^4_{r}\phi^4_{tr}+\phi^4_{t}\phi^4_{tt} drdt\right)^{1/4}\\
 &\leq 16\eps^{1/2}\left(\int \phi^4_{tr}+\phi^4_{tt} drdt\right)^{1/4} \\
 &\leq 16 \eps^{1/2} \left( 2\eps^2 E_{3,q}\left(t\right) +2\eps^2 E_{3,s}\left(t\right) + 48	M_h\left(t\right) \right)^{1/4},
 \end{aligned}
\end{equation}
and
\begin{equation}\label{cha2}
\begin{aligned}
&\left(\int_{0}^{t}\int\left(\phi_{t}\phi_{tr}-\phi_{r}\phi_{tt}\right)^4 dr dt\right)^{1/4}\\
 &\leq 16 \left(\int \phi^4_{r}\phi^4_{tr}+\phi^4_{t}\phi^4_{tt} drdt\right)^{1/4}\\
&\leq 16\eps^{1/2}\left(\int \phi^4_{tr}+\phi^4_{tt} drdt\right)^{1/4} \\
&\leq 16 \eps^{1/2} \left( 2\eps^2 E_{3,q}\left(t\right) +2\eps^2 E_{3,s}\left(t\right) + 48	M_h\left(t\right) \right)^{1/4}.
\end{aligned}
\end{equation}

It is easy to know:
\begin{equation}
\begin{aligned}
\int_{0}^{t}\int \lvert Q^2 \rvert drdt &\leq 4M\left(t\right)+4\eps^{1/2}M^{1/2}_0\left(t\right)M^{1/2}_0\left(t\right)\\
&+2\eps^{3/2}M^{1/2}_0\left(t\right)M^{1/2}_0\left(t\right)+2\eps M\left(t\right)\\
&\leq 4\eps^2+4\eps^{5/2}+2\eps^{7/2}+2\eps^{3}.
\end{aligned}
\end{equation}

Combining  \eqref{s41}, \eqref{s42}, \eqref{cha1} and \eqref{cha2},  we can use h\"oder inequality to obtain:
\begin{equation}
	\begin{aligned}
		&\int_{0}^{t}\int \lvert T^1 \rvert  drdt \leq 8\zeta^{1/2}_1\left(t\right)\gamma^{1/2}_1\left(t\right)+16M^{1/2}_h\left(t\right)\gamma^{1/2}_1\left(t\right) \\
		&+16\eps M^{1/2}_h\left(t\right)\gamma^{1/2}_1\left(t\right)\\  
	    & +16 \eps^{3/2} \left( 2\eps^2 E_{3,q}\left(t\right) +2\eps^2 E_{3,s}\left(t\right) + 48	M_h\left(t\right) \right)^{1/2} \gamma^{1/2}_1\left(t\right)\\
		&+ 256\eps^{3/2}\left( 2\eps^2 E_{3,q}\left(t\right) +2\eps^2 E_{3,s}\left(t\right) + 48	M_h\left(t\right) \right)^{1/2}
		\left(\gamma^{1/2}_1\left(t\right)+\eps^{1/2}\gamma^{1/2}_1\left(t\right)\right)\\
	    &+\left( 2\eps^2 E_{3,q}\left(t\right) + 24	M_h\left(t\right)\right)^{1/2}\left(\gamma^{1/2}_1\left(t\right)+\eps^{1/2}\gamma^{1/2}_1\left(t\right)\right)\\
		&+\left( 2\eps^2 E_{3,s}\left(t\right) + 24	M_h\left(t\right)\right)^{1/2}\left(\gamma^{1/2}_1\left(t\right)+\eps^{1/2}\gamma^{1/2}_1\left(t\right)\right)\\
		&+2M^{1/2}_h\left(t\right)\left(\gamma^{1/2}_1\left(t\right)+\eps^{1/2}\gamma^{1/2}_1\left(t\right)\right),
	\end{aligned}
\end{equation}

\begin{equation}
\begin{aligned}
&\zeta_1\left(t\right) \leq \zeta_2 \left(t\right) +2\eps  M^{1/2}_h\left(t\right)\left(2\eps^2 E_{3,q}\left(t\right) + 24	M_h\left(t\right)\right)^{1/2}\\
&+2\eps M^{1/2}_h\left(t\right)\left(2\eps^2 E_{3,q}\left(t\right) + 24	M_h\left(t\right)\right)^{1/4}\left(2\eps^2 E_{3,s}\left(t\right) + 24	M_h\left(t\right)\right)^{1/4}\\
&+48 \eps^{9/8} M^{1/2}_h\left(t\right)\left(2\eps^2 E_{3,s}\left(t\right) + 24	M_h\left(t\right)\right)^{1/4}
  \left( 2\eps^2 E_{3,q}\left(t\right) +2\eps^2 E_{3,s}\left(t\right) + 48	M_h\left(t\right) \right)^{1/4}\\
&+48 \eps^{9/8} M^{1/2}_h\left(t\right)\left(2\eps^2 E_{3,q}\left(t\right) + 24	M_h\left(t\right)\right)^{1/4}
 \left( 2\eps^2 E_{3,q}\left(t\right) +2\eps^2 E_{3,s}\left(t\right) + 48	M_h\left(t\right) \right)^{1/4}.
\end{aligned}
\end{equation}

 From \eqref{zch} and combining \eqref{wqes}, \eqref{s41}, \eqref{s42}, \eqref{cha1} and \eqref{cha2},  we can use h\"oder inequality to obtain:
\begin{equation*}
	\begin{aligned}
	&\gamma_1\left(t\right) + \frac{1}{8} M_{h,1}\left(t\right) \leq \gamma_2\left(t\right) + \frac{1}{2} M^{1/2}_{h,1}\left(t\right)\gamma^{1/2}_1\left(t\right)+\frac{1}{2}\eps M^{1/2}_{h,1}\left(t\right)\gamma^{1/2}_1\left(t\right)\\
	&+\eps E^{1/4}_{3,q}\left(t\right)E^{1/4}_{3,s}\left(t\right)\eta_1^{1/2}\left(t\right)\gamma^{1/2}_1\left(t\right)+\eps E^{1/2}_{3,q}\left(t\right)\eta_1^{1/2}\left(t\right)\gamma^{1/2}_1\left(t\right)\\
	&+3 \eps E^{1/2}_{3,s}\left(t\right)\eta_1^{1/2}\left(t\right)\gamma^{1/2}_1\left(t\right)+3\eps E^{1/4}_{3,q}\left(t\right)E^{1/4}_{3,s}\left(t\right)\eta_1^{1/2}\left(t\right)\gamma^{1/2}_1\left(t\right)\\
	&+3\eps E^{1/4}_{3,q}\left(t\right)E^{1/4}_{3,s}\left(t\right)\eta_1^{1/2}\left(t\right)\gamma^{1/2}_1\left(t\right)+3 \eps E^{1/2}_{3,q}\left(t\right)\eta_1^{1/2}\left(t\right)\gamma^{1/2}_1\left(t\right)\\
	&+4 \eps^{1/2} \left( 2\eps^2 E_{3,q}\left(t\right) +2\eps^2 E_{3,s}\left(t\right) + 48	M_h\left(t\right) \right)^{1/4}
    \left(2\eps^2 E_{3,q}\left(t\right) + 24	M_h\left(t\right)\right)^{1/4} \eps^{1/2}\gamma^{1/2}_1\left(t\right)\\
    &+16 \eps^{1/2} \left( 2\eps^2 E_{3,q}\left(t\right) +2\eps^2 E_{3,s}\left(t\right) + 48	M_h\left(t\right) \right)^{1/4}
    \left(2\eps^2 E_{3,q}\left(t\right) + 24	M_h\left(t\right)\right)^{1/4} \eps^{3/2}\gamma^{1/2}_1\left(t\right)\\
     &+12 \eps^{1/2} \left( 2\eps^2 E_{3,q}\left(t\right) +2\eps^2 E_{3,s}\left(t\right) + 48	M_h\left(t\right) \right)^{1/4}
    \left(2\eps^2 E_{3,s}\left(t\right) + 24	M_h\left(t\right)\right)^{1/4} \eps^{1/2}\gamma^{1/2}_1\left(t\right)\\
    &+12 \eps^{1/2} \left( 2\eps^2 E_{3,q}\left(t\right) +2\eps^2 E_{3,s}\left(t\right) + 48	M_h\left(t\right) \right)^{1/4}
    \left(2\eps^2 E_{3,s}\left(t\right) + 24	M_h\left(t\right)\right)^{1/4} \eps^{3/2}\gamma^{1/2}_1\left(t\right)\\
     &+12 \eps^{1/2} \left( 2\eps^2 E_{3,q}\left(t\right) +2\eps^2 E_{3,s}\left(t\right) + 48	M_h\left(t\right) \right)^{1/4} \left(2\eps^2 E_{3,q}\left(t\right) + 24	M_h\left(t\right)\right)^{1/4} \eps^{3/2}\gamma^{1/2}_1\left(t\right)\\
\end{aligned}
\end{equation*}
\begin{equation}
\begin{aligned} 
    &+\eps^{3/2}M^{1/2}_{h,1}\left(t\right)\\
	&+16 \eps \left( 2\eps^2 E_{3,q}\left(t\right) +2\eps^2 E_{3,s}\left(t\right) + 48	M_h\left(t\right) \right)^{1/2}  M^{1/2}_{h,1}\left(t\right)\\
	&+\eps  M^{1/2}_{h,1}\left(t\right)\left(2\eps^2 E_{3,q}\left(t\right) + 24	M_h\left(t\right)\right)^{1/4}\left(2\eps^2 E_{3,s}\left(t\right) + 24	M_h\left(t\right)\right)^{1/4}\\
	&+\eps  M^{1/2}_{h,1}\left(t\right)\left(2\eps^2 E_{3,q}\left(t\right) + 24	M_h\left(t\right)\right)^{1/2}\\
	&+4 \eps^2 \left( 2\eps^2 E_{3,q}\left(t\right) +2\eps^2 E_{3,s}\left(t\right) + 48	M_h\left(t\right) \right)^{1/2} M^{1/2}_{h,1}\left(t\right).\\
	\end{aligned}
\end{equation}

Using the div-curl lemma \ref{dc}, we have:
\begin{equation}
\begin{aligned}
\zeta _2\left(t\right)&\lesssim\left(\hat{E}_1\left(0\right)  + \sup\limits_{0\leq t \leq T} \hat{E}_1\left(t\right) + \int_{0}^{t}\int \lvert Q^2 \rvert drdt\right )\\
&\cdot\left(\hat{E}_3\left(0\right) +\sup\limits_{0\leq t \leq T}\hat{E}_3\left(t\right)+\int_{0}^{t}\int \lvert T^1 \rvert drdt \right)\\
&\lesssim\left(\eps^2  + \sup\limits_{0\leq t \leq T} \hat{E}_1\left(t\right) + \int_{0}^{t}\int \lvert Q^2 \rvert drdt\right )\\
&\cdot\left(\hat{E}_3\left(0\right) +\sup\limits_{0\leq t \leq T}\hat{E}_3\left(t\right)+\int_{0}^{t}\int \lvert T^1 \rvert drdt \right),
\end{aligned}
\end{equation}

\begin{equation}
\begin{aligned}
&\gamma_2\left(t\right)\\
&\lesssim\left(\hat{E}_3\left(0\right)  + \sup\limits_{0\leq t \leq T} \hat{E}_3\left(t\right) + \int_{0}^{t}\int \lvert T^1 \rvert drdt\right)\\
&\cdot\left(\hat{E}_2\left(0\right)+\hat{E}^{1/2}_1\left(0\right)\hat{E}^{1/2}_2\left(0\right) +\sup\limits_{0\leq t \leq T}\left(\hat{E}_2\left(t\right)+\hat{E}^{1/2}_1\left(t\right)\hat{E}^{1/2}_2\left(t\right)\right)+\int_{0}^{t}\int \lvert P^1 \rvert drdt\right)\\
&\lesssim\left(\hat{E}_3\left(0\right)  + \sup\limits_{0\leq t \leq T} \hat{E}_3\left(t\right) + \int_{0}^{t}\int \lvert T^1 \rvert drdt \right)\\
&\cdot\left(\eps^2 +\sup\limits_{0\leq t \leq T}\left(\hat{E}_2\left(t\right)+\hat{E}^{1/2}_1\left(t\right)\hat{E}^{1/2}_2\left(t\right)\right)+\int_{0}^{t}\int \lvert P^1 \rvert drdt \right).
\end{aligned}
\end{equation}

Noting that
\begin{equation}
\begin{aligned}
&\phi^2_{r}\phi^2_{ttt}\lesssim\left(\phi_{r}\phi_{ttt}-\phi_{t}\phi_{ttr}\right)^2 + \phi^2_{t}\phi^2_{ttr}\\
&\phi^2_{r}\phi^2_{ttr}\lesssim\left(\phi_{r}\phi_{ttr}-\phi_{t}\phi_{ttt}\right)^2 + \phi^2_{t}\phi^2_{ttt},
\end{aligned}
\end{equation}
so we have:
\begin{equation}
\begin{aligned}
M_{h,2}\left(t\right) &\lesssim M_{h,1}\left(t\right) + \zeta_1\left(t\right) .
\end{aligned}
\end{equation}
By definition, we know:
\begin{equation}
M_h\left(t\right) = M_{h,1}\left(t\right)+M_{h,2}\left(t\right) .
\end{equation}

Based above, we can obtain the following estimates
\begin{equation}
	E_3\left(t\right) \leq 16 E_3\left(0\right)
\end{equation}
by the continuous induction method and  standard bootstrap arguments.

Accoring to the local existence results of Smith-Tataru \cite{Smith 2005}, we know the extremal hypersurface equation \eqref{equ} exists a global solution.

\section{Stability}\label{88}
In this section, we use the homotopy method to prove the stability of solution for the system \eqref{equs}.

We define $ u:= u \left(\lambda; t, x\right)$ is the solution for the following system
\begin{equation}\label{equtl}
\left\{
\begin{aligned}
& (\frac{r u_t }{\Delta_\lambda^{1/2} })_t - (\frac{r u_r }{\Delta_\lambda^{1/2}})_r=0,\\
& t=0: u=\lambda\phi_0+\left(1-\lambda\right)\widetilde{\phi}_0, u_{t}=\phi_1 + \left(1-\lambda\right)\widetilde{\phi}_1,
\end{aligned}
\right. 
\end{equation}
where
\begin{equation*}
	\Delta_\lambda :=\left(1+u_r^2-u_t^2\right).
\end{equation*}

From Newton-Leibniz formula, we know:
\begin{equation}\label{tl}
	\phi-\widetilde{\phi}= \int_{0}^{1} \frac{\partial u}{\partial \lambda} d\lambda
\end{equation}. 

So we only need to estimate the $H^1$ norm for $ \frac{\partial u}{\partial \lambda} $.

Just like in section \ref{bal}, we establish some balance law for \eqref{equtl}.

We multiply both sides of \eqref{equtl} by $\frac{\phi_{r}}{r^2}$, we have:
\begin{equation}
\frac{u_{r}}{r^2}\left[\left(\frac{ru_{t}}{\Delta_\lambda^{1/2}}\right)_t-\left(\frac{ru_{r}}{\Delta_\lambda^{1/2}}\right)_r\right]=0.
\end{equation}
It is not difficult to obtain:
\begin{equation}\label{phu1}
\left(\frac{u_r u_{t}}{r\Delta_\lambda^{1/2}}\right)_t -\frac{1}{2} \left(\frac{u^2_{t}+u^2_{r}}{r\Delta_\lambda^{1/2}}\right)_r=Q_\lambda^2,
\end{equation}
where 
\begin{equation*}
Q_\lambda^2=-\frac{1}{2}u^2_{t}\left(\frac{1}{r\Delta_\lambda^{1/2}}\right)_r+\frac{1}{2}u^2_r\left(\frac{1}{r\Delta_\lambda^{1/2}}\right)_r+\frac{2u^2_{r}}{r^2\Delta_\lambda^{1/2}}.
\end{equation*}

Taking the first derivative of the equation \eqref{equtl} with respect to $\lambda$, we can get:
\begin{equation}\label{1-derl}
\left(\frac{r u_{\lambda t}\left(1+u^2_r\right)-r u_t u_r u_{\lambda r}}{\Delta_\lambda^{3/2} }\right)_t - \left(\frac{r u_{\lambda r}\left( 1-\phi^2_t\right) + r u_r u_t u_{\lambda t}}{\Delta_\lambda ^{3/2}}\right)_r = 0 .
\end{equation}

Taking the first derivative of the equation \eqref{equtl} with respect to time, we have:
\begin{equation}\label{1-derrll}
r\left(\frac{u_{tt}\left(1+u^2_{r}\right)-u_{t}u_{r}u_{tr}}{\Delta_\lambda^{3/2}}\right)_t-r\left(\frac{u_{tr}\left(1-u^2_{t}\right)+u_{r}u_{t}u_{tt}}{\Delta_\lambda^{3/2}}\right)_r-\left(\frac{u_{r}}{\Delta_\lambda^{1/2}}\right)_t=0.
\end{equation}

Let us multiply both side of \eqref{1-derl} by $u_{\lambda t}$, we can obtain:
\begin{equation}
u_{\lambda t}\left[\left(\frac{r u_{\lambda t}\left(1+u^2_r\right)-r u_t u_r u_{\lambda r}}{\Delta_\lambda^{3/2} }\right)_t - \left(\frac{r u_{\lambda r}\left( 1-\phi^2_t\right) + r u_r u_t u_{\lambda t}}{\Delta_\lambda ^{3/2}}\right)_r\right] = 0 .
\end{equation}

By some  simple calculation, it is not difficult to see the following equality:
\begin{equation}\label{phu2}
\begin{aligned}
\frac{1}{2} \left(\frac{ru^2_{\lambda t}\left(1+u^2_r\right)+ru^2_{ \lambda r}\left(1-u^2_t\right)}{\Delta_\lambda^{3/2}}\right)_t-\left( \frac{r u_{\lambda t}u _{\lambda r}\left(1-u^2_t\right)+r u_r u_t u^2_{\lambda t}}{\Delta_\lambda ^{3/2}}\right)_r=W_\lambda^1,
\end{aligned}
\end{equation}
where
\begin{equation*}
W_\lambda^1=-\frac{1}{2}u^2_{\lambda t}\left(\frac{r\left(1+ u^2_r\right)}{\Delta_\lambda^{3/2}}\right)_t+u_{\lambda t}u_{\lambda r}\left(\frac{ru_t u_r}{\Delta_\lambda^{3/2}}\right)_t+\frac{1}{2}u^2_{tr}\left(\frac{r\left(1-u^2_t\right)}{\Delta_\lambda^{3/2}}\right)_t.
\end{equation*}

In fact,
\begin{equation}
\begin{aligned}
W_\lambda^1&=
r\frac{9}{2}\frac{1}{\Delta_\lambda^{5/2}}\left(u_{r}u_{\lambda r}-u_{t} u_{\lambda t}\right)\left(u_{tr}u_{\lambda r}\left(1-\phi^2_{t}\right)-u_{tt}u_{\lambda t}\left(1+\phi^2_{r}\right)+2u_{\lambda t}u_{tr}u_{r}u_{t}\right)\\
&+r\frac{9}{2}\frac{1}{\Delta_\lambda^{5/2}}u_{t}u_{\lambda r}\left(1-u^2_{t}\right)\left(u_{\lambda t}u_{tr}-u_{\lambda r}u_{tt}\right)\\
&+r\frac{9}{2}\frac{1}{\Delta_\lambda^{5/2}}u_{r}u_{\lambda t}\left(1+u^2_{r}\right)\left(u_{\lambda r}u_{tt}-u_{\lambda t}u_{tr}\right)\\
&+2r\frac{9}{2}\frac{1}{\Delta_\lambda^{5/2}}
u^2_{t}u_{r}u_{\lambda t}\left(u_{\lambda t}u_{tr}-u_{\lambda r}u_{tt}\right)\\
&+r\left(\frac{1}{\Delta_\lambda^{3/2}}\right)\left(u_{\lambda r}u_{t}-u_{\lambda t}u_{r}\right)\left(u_{\lambda t}u_{tr}-u_{\lambda r}u_{tt}\right).
\end{aligned}
\end{equation}

Let us multiply both side of \eqref{1-derrll} by $\left(u_{tr}+ \frac{u_t}{2r}\right)$, we can obtain:
\begin{equation}
\begin{aligned}
&\left(u_{tr}+ \frac{u_t}{2r}\right)\left[r\left(\frac{u_{tt}\left(1+u^2_{r}\right)-u_{t}u_{r}u_{tr}}{\Delta_\lambda^{3/2}}\right)_t-r\left(\frac{u_{tr}\left(1-u^2_{t}\right)+u_{r}u_{t}u_{tt}}{\Delta_{\lambda}^{3/2}}\right)_r\right]\\
&-\left(u_{tr}+\frac{u_{t}}{2r}\right)\left(\frac{u_{r}}{\Delta_{\lambda}^{1/2}}\right)_t=0.
\end{aligned}
\end{equation}

According to a simple calculation, it is not difficult to see following equality:
\begin{equation}\label{phu3}
\begin{aligned}
&\left(\frac{ru_{tt}u_{tr}\left(1+u^2_{r}\right)-ru_{t}u_{r}u^2_{tr}}{\Delta_\lambda^{3/2}}+\frac{1}{2}\frac{u_{t}u_{tt}\left(1+u^2_{r}\right)-u^2_{t}u_{r}u_{tr}}{\Delta_\lambda^{3/2}}\right)_t\\
&-\left(\frac{1}{2}\frac{ru^2_{tr}\left(1-u^2_{t}\right)+ru^2_{tt}\left(1+u^2_{r}\right)}{\Delta_\lambda^{3/2}}+\frac{1}{2}\frac{u_{t}u_{tr}\left(1-u^2_{t}\right)+u^2_{t}u_{r}u_{tt}}{\Delta_\lambda^{3/2}}+\frac{1}{4}\frac{u^2_{t}}{r\Delta_\lambda^{1/2}}\right)_r\\
&=P_\lambda^1,
\end{aligned}
\end{equation}
where
\begin{equation}
\begin{aligned}
&P_\lambda^1=r\left(\frac{1}{\Delta_\lambda^{3/2}}\right)_r\left(-\frac{u^2_{tt}}{2}\left(1+u^2_{r}\right)+u_{tr}u_{tt}u_{r}u_{t}+\frac{u^2_{tr}}{2}\left(1-u^2_{t}\right)\right)\\
&+\left(\frac{r}{\Delta_\lambda^{3/2}}\right)\left(u_{tt}u_{r}-u_{tr}u_{t}\right)\left(u^2_{tr}-u_{tt}u_{rr}\right)\\
&+\frac{1}{4}\frac{u^2_{t}}{r^2\Delta_\lambda^{1/2}}-\left(\frac{1}{\Delta_\lambda^{1/2}}\right)_r\frac{u^2_{t}}{4r}+\left(\frac{1}{\Delta_\lambda^{1/2}}\right)_t\frac{u_{t}u_{r}}{2r}.
\end{aligned}
\end{equation}

Then we study the ``inner product"  induced by the law of equilibrium like the section \eqref{55}.

  Considering the ``inner product" induced by the law of equilibrium \eqref{phu1} and \eqref{phu2}, we have:
\begin{equation}
C^{\lambda}\left(u\left(t\right)\right):=
\begin{pmatrix}
a &b \\
c           & d

\end{pmatrix},
\end{equation}
where
\begin{equation*}
a:= \frac{u_r u_{t}}{r\Delta_\lambda^{1/2}},
\end{equation*}

\begin{equation*}
b:=\frac{u^2_{t}+u^2_{r}}{r\Delta_\lambda^{1/2}},
\end{equation*}

\begin{equation*}
c:= \frac{1}{2}\frac{ru^2_{\lambda t}\left(1+u^2_r\right)+ru^2_{ \lambda r}\left(1-u^2_t\right)}{\Delta_\lambda^{3/2}},
\end{equation*}

\begin{equation*}
d:= \frac{r u_{\lambda t}u _{\lambda r}\left(1-u^2_t\right)+r u_r u_t u^2_{\lambda t}}{\Delta_\lambda ^{3/2}}.
\end{equation*}

By some simple calcaution, we have:
\begin{equation}
\begin{aligned}
det C^{\lambda}\left(u\left(t\right)\right)&=\frac{1}{4\Delta_\lambda^2}\left(u_{t}u_{\lambda t}-u_{r}u_{\lambda r}\right)^2+\frac{1}{4\Delta_\lambda^2}\left(u_{t}u_{\lambda r}-u_{r}u_{\lambda t}\right)^2\\
&+\frac{1}{4\Delta_\lambda^2}\left(u^2_{t}+u^2_{r}\right)\left( u_r u_{\lambda t}- u_{t}u_{\lambda r}\right)^2\\
&+\frac{1}{2\Delta_\lambda^2}\left(u^2_{t}+u^2_{r}\right)\left(u_{\lambda t}u_{r}-u_{t}u_{\lambda r}\right)\left(u_{\lambda r}u_{r}-u_{\lambda t}u_{t}\right)\\
&+\frac{1}{2\Delta_\lambda^2}\left(u^2_{t}+u^2_{r}\right)\left(u_{\lambda t}u_{r}-u_{t}u_{\lambda r}\right)\left(u_{\lambda t}u_{t}-u_{\lambda r}u_{r}\right)\\
&+\frac{1}{2\Delta_\lambda^2}u^2_{t}\left(u_{\lambda r}u_{t}-u_{\lambda t}u_{r}\right)\left(u_{\lambda t}u_{r}-u_{t}u_{\lambda r}\right)\\
&+\frac{1}{2\Delta_\lambda^2}u_{t}u_{r}\left(u_{\lambda r}u_{r}-u_{t}u_{\lambda t}\right)\left(u_{\lambda t}u_{r}-u_{t}u_{\lambda r}\right).\\
\end{aligned}
\end{equation}

Considering the ``inner product" induced by the law of equilibrium \eqref{phu2} and \eqref{phu3}, we have:

\begin{equation}
D^{\lambda}\left(u\left(t\right)\right):=
\begin{pmatrix}
a & b\\
c^1+c^2            & d^1+d^2+d^3 
\end{pmatrix}
\end{equation}
where
\begin{equation*}
a:=\frac{1}{2}\frac{ru^2_{\lambda t}\left(1+u^2_r\right)+ru^2_{ \lambda r}\left(1-u^2_t\right)}{\Delta_\lambda^{3/2}},
\end{equation*}

\begin{equation*}
b:=\frac{r u_{\lambda t}u _{\lambda r}\left(1-u^2_t\right)+r u_r u_t u^2_{\lambda t}}{\Delta_\lambda ^{3/2}},
\end{equation*}

\begin{equation*}
c^1:= \frac{ru_{tt}u_{tr}\left(1+u^2_{r}\right)-ru_{t}u_{r}u^2_{tr}}{\Delta_\lambda^{3/2}},
\end{equation*}

\begin{equation*}
c^2:= \frac{1}{2}\frac{u_{t}u_{tt}\left(1+u^2_{r}\right)-u^2_{t}u_{r}u_{tr}}{\Delta_\lambda^{3/2}},
\end{equation*}

\begin{equation*}
d^1:=\frac{1}{2}\frac{ru^2_{tr}\left(1-u^2_{t}\right)+ru^2_{tt}\left(1+u^2_{r}\right)}{\Delta_\lambda^{3/2}},
\end{equation*}

\begin{equation*}
d^2:=\frac{1}{2}\frac{u_{t}u_{tr}\left(1-u^2_{t}\right)+u^2_{t}u_{r}u_{tt}}{\Delta_\lambda^{3/2}},
\end{equation*}

\begin{equation*}
d^3:=\frac{1}{4}\frac{u^2_{t}}{r\Delta_\lambda^{1/2}}.
\end{equation*}

\begin{equation}\label{zh}
\begin{aligned}
&det D^{\lambda}\left(u\left(t\right)\right)= ad^1+ad^2+ad^3-bc^1-bc^2\\
&= det D^{\lambda}_m\left(u\left(t\right)\right)\\
&+\frac{r}{4\Delta_\lambda^3}u_{t}u_{\lambda t}\left(1-u^2_{t}\right)\left(1+u^2_{r}\right)\left(u_{\lambda t}u_{tr}-u_{\lambda r}u_{tt}\right)\\
&+\frac{r}{4\Delta_\lambda^3}u_{t}u_{\lambda r}\left(1-u^2_{t}\right)\left(u_{\lambda r}u_{tr}\left(1-u^2_{t}\right)-u_{\lambda t}u_{tt}\left(1+u^2_{r}\right)+2u_{\lambda r}u_{tr}u_{t}u_{r}\right)\\
&+\frac{r}{4\Delta_\lambda^3}u^2_{t}u_{r}u_{\lambda r}\left(1-u^2_{t}\right)\left(u_{\lambda r}u_{tt}-u_{\lambda t}u_{tr}\right)\\
&+\frac{r}{4\Delta_\lambda^3}u^2_{t}u_{r}u_{\lambda t}\left(u_{\lambda r}u_{tr}\left(1-u^2_{t}\right)-u_{\lambda t}u_{tt}\left(1+u^2_{r}\right)+2u_{\lambda t}u_{tr}u_{t}u_{r}\right)\\
&+\frac{1}{8\Delta_\lambda^2}\left(u^2_{t}u^2_{\lambda t}\left(1+u^2_{r}\right)+u^2_{t}u^2_{\lambda r}\left(1-u^2_{t}\right)\right),\\
\end{aligned}
\end{equation}
where
\begin{equation*}
\begin{aligned}
det D^{\lambda}_m \left(u\left(t\right)\right):=& \frac{r^2}{4\Delta_\lambda ^3}\left(1+u^2_{r}\right)\left(1-u^2_{t}\right)\left(u_{\lambda t}u_{tr}-u_{tt}u_{\lambda r}\right)^2\\
&+\frac{r^2}{4\Delta_\lambda^3}\left(u_{\lambda r}u_{tr}\left(1-u^2_{t}\right)-u_{\lambda t}u_{tt}\left(1+u^2_{r}\right)+2u_{\lambda t}u_{tr}u_{t}u_{r}\right).
\end{aligned}
\end{equation*}

Since $u$ is a solution to the \eqref{equtl}, we can repeat exactly what we did in section \ref{ei}. 

So we have:
\begin{equation}\label{smee}
\begin{aligned}
&E_{1,u}\left(t\right):=E_{1}\left(u\left(t\right)\right)\lesssim \eps^2,\\
&E_{2,u}\left(t\right):=E_{2}\left(u\left(t\right)\right)\lesssim \eps^2,\\
&M_{0,u}\left(t\right):=M_{0}\left(u\left(t\right)\right)\lesssim \eps^2,\\
&M_{u}\left(t\right):=M\left(u\left(t\right)\right)\lesssim\eps^2,\\
&\xi_{1,u}\left(t\right):=\xi_{1}\left(u\left(t\right)\right)\lesssim \eps^4,\\
&\xi_{2,u}\left(t\right):=\xi_{2}\left(u\left(t\right)\right)\lesssim \eps^4,\\
&\eta_{1,u}\left(t\right):=\eta_{1}\left(u\left(t\right)\right)\lesssim \eps^4,\\
&\eta_{2,u}\left(t\right):=\eta_{2}\left(u\left(t\right)\right)\lesssim \eps^4,
\end{aligned}
\end{equation}
where the notations $E_{1}\left(u\left(t\right)\right)$, $E_{2}\left(u\left(t\right)\right)$, $M_{0}\left(u\left(t\right)\right)$, $M\left(u\left(t\right)\right)$, $\xi_{1}\left(u\left(t\right)\right)$, $\xi_{2}\left(u\left(t\right)\right)$, $\eta_{1}\left(u\left(t\right)\right)$,  and $\eta_{2}\left(u\left(t\right)\right)$  are defined as same as section \ref{NP}.

Moreover, we also define:
\begin{equation*}
	E_{\lambda} \left(t\right):= \int r \left(u^2_{\lambda t} + u^2_{\lambda r} \right) dr
\end{equation*}

\begin{equation*}
\hat{E}_{\lambda}\left(t\right):=\int\frac{1}{2}\frac{u^2_{tr}\left(1-u^2_{t}\right)+u^2_{tt}\left(1+u^2_{r}\right)}{\Delta_{\lambda}^{3/2}} r dr,
\end{equation*}

\begin{equation*}
M_{0,\lambda }\left(t\right):=\int_{0}^{t}\int u^2_{\lambda t}u^2_{r}+u^2_{\lambda r}u^2_{t}+u^2_{\lambda t}u^2_{t}+u^2_{\lambda r}u^2_{r} drds,
\end{equation*}

\begin{equation*}
M_{0,\lambda,1}\left(t\right):=\int_{0}^{t}\int u^2_{\lambda t}u^2_{t}+u^2_{\lambda r}u^2_{t} drds,
\end{equation*}

\begin{equation*}
M_{0,\lambda,2}\left(t\right):=\int_{0}^{t}\int u^2_{\lambda t} u ^2_{r} + u^2_{\lambda r}u^2_{r} drds,
\end{equation*}

\begin{equation*}
\zeta_\lambda\left(t\right):=\int^t_0\int det C^{\lambda } \left(u\left(t\right)\right) drds,
\end{equation*}

\begin{equation*}
\gamma_{1,\lambda}\left(t\right):=\int^t_0\int det D^{\lambda}_m\left(u \left(t\right)\right) drds,
\end{equation*}

\begin{equation*}
\gamma_{2,\lambda}\left(t\right):=\int^t_0\int det D^{\lambda}\left(u\left(t\right)\right) drds,
\end{equation*}

We note:
\begin{equation}
	E_{\lambda} \left(t\right)\leq 4 \hat{E}_{\lambda}\left(t\right).
\end{equation}

According to \eqref{phu1}, we have:
\begin{equation}
\begin{aligned}
\hat{E}_{\lambda}\left(t\right)&\leq \hat{E}_{\lambda}\left(0\right) + \int_{0}^{t}\int\lvert W_{\lambda}^1\rvert drdt.
\end{aligned}
\end{equation}

It is easy to know:
\begin{equation}
\begin{aligned}
\int_{0}^{t}\int \lvert Q_\lambda^2 \rvert drdt &\leq 4M_u\left(t\right)+4\eps^{1/2}M^{1/2}_{0,u}\left(t\right)M^{1/2}_{0,u}\left(t\right)\\
&+2\eps^{3/2}M^{1/2}_{0,u}\left(t\right)M^{1/2}_{0,u}\left(t\right)+2\eps M_u\left(t\right)\\
&\leq 4\eps^2+4\eps^{5/2}+2\eps^{7/2}+2\eps^{3},
\end{aligned}
\end{equation}
and
\begin{equation}
\begin{aligned}
\int_{0}^{t}\int \lvert P_\lambda^1 \rvert drdt\leq& M_{1,u}\left(t\right) +8M^{1/2}_{0,u}\left(t\right)\xi^{1/2}_{1,u}\left(t\right)+8M_u^{1/2}\left(t\right)\xi^{1/2}_{1,u}\left(t\right)+16\eta^{1/2}_{1,u}\left(t\right)\xi^{1/2}_{1,u}\left(t\right)\\
&+16\eta_{1,u}^{1/2}\left(t\right)M^{1/2}_{0,u}\left(t\right)+2\eta^{1/2}_{1,u}M_u^{1/2}\left(t\right)+\eps M_{0,u}^{1/2}\left(t\right)M_u^{1/2}\left(t\right)+\eps M_u\left(t\right).
\end{aligned}
\end{equation}

In addition,

\begin{equation}
\begin{aligned}
\int_{0}^{t}\int \lvert W_{\lambda}^1 \rvert  drdt \leq& 8\zeta^{1/2}_\lambda\left(t\right)\gamma^{1/2}_{1,\lambda}\left(t\right)+16M^{1/2}_{0,\lambda}\left(t\right)\gamma^{1/2}_{1,\lambda}\left(t\right) \\
&+16\eps M^{1/2}_{0,\lambda}\left(t\right)\gamma^{1/2}_{1,\lambda}\left(t\right).
\end{aligned}
\end{equation}

From \eqref{zh}, we know:
\begin{equation}
\begin{aligned}
\gamma_{1,\lambda}\left(t\right) + \frac{1}{8} M_{0,\lambda,1}\left(t\right) \leq \gamma_2\left(t\right) + \frac{1}{2} M^{1/2}_{0,\lambda,1}\left(t\right)\gamma^{1/2}_{1,\lambda}\left(t\right)+\frac{1}{2}\eps M^{1/2}_{0,\lambda,1}\left(t\right)\gamma^{1/2}_{1,\lambda}\left(t\right).
\end{aligned}
\end{equation}

Using the div-curl lemma \ref{dc}, we have:
\begin{equation}
\begin{aligned}
\zeta _\lambda\left(t\right)&\lesssim\left(\hat{E}_{1,u}\left(0\right)  + \sup\limits_{0\leq t \leq T} \hat{E}_{1,u}\left(t\right) + \int_{0}^{t}\int \lvert Q_\lambda^2 \rvert drdt\right )\\
&\cdot\left(\hat{E}_\lambda \left(0\right) +\sup\limits_{0\leq t \leq T}\hat{E}_\lambda \left(t\right)+\int_{0}^{t}\int \lvert W_\lambda^1 \rvert drdt \right)\\
&\lesssim\left(\eps^2  + \sup\limits_{0\leq t \leq T} \hat{E}_{1,u}\left(t\right) + \int_{0}^{t}\int \lvert Q_\lambda^2 \rvert drdt\right )\\
&\cdot\left(\hat{E}_\lambda \left(0\right) +\sup\limits_{0\leq t \leq T}\hat{E}_\lambda \left(t\right)+\int_{0}^{t}\int \lvert W_\lambda^1 \rvert drdt \right),\\
\end{aligned}
\end{equation}

\begin{equation}
\begin{aligned}
&\gamma_{2,\lambda}\left(t\right)\lesssim\left(\hat{E}_\lambda\left(0\right)  + \sup\limits_{0\leq t \leq T} \hat{E}_\lambda\left(t\right) + \int_{0}^{t}\int \lvert W^1_\lambda \rvert drdt\right)\\
&\cdot\left(\hat{E}_{2,u}\left(0\right)+\hat{E}^{1/2}_{1,u}\left(0\right)\hat{E}^{1/2}_{2,u}\left(0\right) +\sup\limits_{0\leq t \leq T}\left(\hat{E}_{2,u}\left(t\right)+\hat{E}^{1/2}_{1,u}\left(t\right)\hat{E}^{1/2}_{2,u}\left(t\right)\right)+\int_{0}^{t}\int \lvert P_\lambda^1 \rvert drdt\right)\\
&\lesssim\left(\hat{E}_\lambda \left(0\right)  + \sup\limits_{0\leq t \leq T} \hat{E}_\lambda \left(t\right) + \int_{0}^{t}\int \lvert W_\lambda^1 \rvert drdt \right)\\
&\cdot\left(\eps^2 +\sup\limits_{0\leq t \leq T}\left(\hat{E}_2\left(t\right)+\hat{E}^{1/2}_{1,u}\left(t\right)\hat{E}^{1/2}_{2,u}\left(t\right)\right)+\int_{0}^{t}\int \lvert P_\lambda^1 \rvert drdt \right).
\end{aligned}
\end{equation}

Noting that
\begin{equation}
\begin{aligned}
&u^2_{r}u^2_{\lambda t}\lesssim\left(u_{r}u_{\lambda t}-u_{t}u_{\lambda r}\right)^2 + u^2_{t}u^2_{\lambda r}\\
&u^2_{r}u^2_{\lambda r}\lesssim\left(u_{r}u_{\lambda r}-u_{t}u_{\lambda t}\right)^2 + u^2_{t}u^2_{\lambda t},
\end{aligned}
\end{equation}
so we have:
\begin{equation}
\begin{aligned}
M_{0,\lambda,2}\left(t\right) &\lesssim M_{0,\lambda,1}\left(t\right) + \zeta_\lambda \left(t\right) .
\end{aligned}
\end{equation}
By definition, we know:
\begin{equation}
M_h\left(t\right) = M_{h,1}\left(t\right)+M_{h,2}\left(t\right) .
\end{equation}

Based above, we can obtain the following estimates
\begin{equation}\label{st}
E_\lambda\left(t\right) \leq 16 E_\lambda\left(0\right)
\end{equation}
by the continuous induction method and  standard bootstrap arguments.

From \eqref{st}, we get:
\begin{equation}\label{tds}
	\|\frac{\partial u}{\partial \lambda}\|_{H^1_{rad}}\leq 16 \left( \|\phi_0-\widetilde{\phi}_0\|_{H^1_{rad}} + \|\phi_1-\widetilde{\phi}_1\|_{L^2_{rad}}\right) .
\end{equation}

Combining \eqref{tl} and \eqref{tds}, we can obtain:
	\begin{equation}
\| \phi_{t}\left(t\right)-\widetilde{\phi}_t\left(t\right)\|_{L^2_{rad}} + \|\phi\left(t\right)-\widetilde{\phi}\left(t\right)\|_{H^1_{rad}} \leq 16 \left( \|\phi_0-\widetilde{\phi}_0\|_{H^1_{rad}} + \|\phi_1-\widetilde{\phi}_1\|_{L^2_{rad}}\right)
\end{equation}
for any $t\in \left(0, \infty\right)$.

\section{main results}\label{99}
In this section, we will  prove the  theorem \ref{1.3}, theorem \ref{1.4}, and theorem \ref{1.5}.

\begin{proof}[theorem \ref{1.3}]
	
	Combining the global existence  in section \ref{77} and stability in section \ref{88}, we have proven the theorem \ref{1.3}.
\end{proof}

\begin{proof}[theorem \ref{1.4}]
	
	 When then initial data  $\|\left(\phi_0, \phi_1\right)\|_{H_{rad}^2\left(\mathbb{R}^2\right) \times H_{rad}^{1}\left(\mathbb{R}^2\right)} \leq \eps$ which is sufficiently small, we can  modify the initial data as follow:
	 \begin{equation}
	 	\left(J_{\frac{1}{n}}\phi_0, J_{\frac{1}{n}}\phi_1\right) \in H^s\left(\mathbb{R}^2\right) \times H^{s-1} \left(\mathbb{R}^2\right) \quad \forall s\ge 3,
	  \end{equation}
	  where the $J_{\frac{1}{n}}$ is the standard modify  operator.

From  the theorem \ref{1.3}, we know for any $n$, there exists a 
$\phi_{n} \in C^0\left(\left[0, T\right], H^s\left(\mathbb{R}^2\right)\right) \cap C^1\left(\left[0, T\right], H^{s-1}\left(\mathbb{R}^2\right)\right)$  with $s\ge 3$ for any $T>0$. Besides,  $\{\phi_{n}\}$ is uniformly bounded in $C^0\left(\left[0, T\right], H_{rad}^2\left(\mathbb{R}^2\right)\right)$ $\cap C^1\left(\left[0, T\right], H_{rad}^{1}\left(\mathbb{R}^2\right)\right)$.

So we know: 
\begin{equation}
	\phi_{n} \stackrel{\mathrm{w}^*}{\to} \phi \in L^{\infty}\left(\left[0, T\right], H_{rad}^2\left(\mathbb{R}^2\right)\right)\cap W^{1,\infty}\left(\left[0, T\right], H_{rad}^{1}\left(\mathbb{R}^2\right)\right) 
\end{equation}
in the weekly $-*$ sense.

Based above and combining the stability in section \ref{88}, we have proven the theorem \ref{1.4}.

\end{proof}

\begin{proof}[theorem \ref{1.5}]
	For any initial data $\left(\phi_0, \phi_1\right) \in H_{rad}^2\left(\mathbb{R}^2\right) \times H_{rad}^{1} \left(\mathbb{R}^2\right)$, we can find a $\delta_1$ to make the integral
	\begin{equation}
		\int_0^{2\delta_1} r\phi^2_{0rr}+\frac{\phi^2_{0r}}{r} + r\phi^2_{_{1r}} + \frac{\phi^2_1}{r} dr \leq \eps^2 .
	\end{equation}

By finite propagation speed of waves and applying theorem \ref{1.4}, we know when $0\leq t\leq 2\delta_1$, $0 \leq  r \leq 2 \delta_1$ and $ r+t \leq 2\delta_1$, the equation \eqref{equ} exists a local solution. Particularly, when $ 0\leq t\leq \frac{\delta_1}{2}$,  $0 \leq  r \leq 2 \delta_1$ and  $t \leq -r+2\delta_1$, the equation \eqref{equ} exists a local solution. Considering the region $r\ge \delta_1$, this case essentially becomes one-dimensional. So according to the results in Zhou \cite{Zhou 1996}, we know  there exists a $\delta_2$ and when $r\ge \delta_1$, $t\leq \delta_2$, $t\leq r-\delta_1$, the equation \eqref{equ} exists a local solution. Then we choose $\delta$:= min$\{\frac{\delta_1}{2}, \delta_2\}$ which may depend on the $\|\left(\phi_0, \phi_1\right)\|_{H_{rad}^2\left(\mathbb{R}^2\right) \times H_{rad}^{1}\left(\mathbb{R}^2\right)}$ and the profiles of initial data $\left(\phi_0, \phi_1\right) $. We know the equation \eqref{equ} exists a local solution $\phi \in L^{\infty}\left(\left[0, T^*\right], H_{rad}^2\left(\mathbb{R}^2\right)\right) \cap  W^{1,\infty}\left(\left[0, T^*\right], H_{rad}^{1}\left(\mathbb{R}^2\right)\right)$ for any $T^* \leq \delta$. So we have proven the theorem \ref{1.5}.
\end{proof}

\par{\bf Acknowledgement}:

This work was supported by the National Natural Science Foundation of China (No. 12171097), the Key Laboratory of Mathematics for Nonlinear Sciences (Fudan University), Ministry of Education of China, P.R.China. Shanghai Key Laboratory for Contemporary Applied Mathematics, School of Mathematical Sciences, Fudan University, P.R. China, and by Shanghai Science and Technology Program [Project No. 21JC1400600].

\par{\bf Author declarations}:

{\bf Conflict of interest}:

The authors have no condlicts of interest to disclose.

{\bf Data availability}:

Data sharing is not applicable to this article as no new data were created or analyzed in this study.

\end{document}